\documentclass[reqno]{amsart}
\usepackage[T1]{fontenc}
\usepackage[utf8]{inputenc}
\usepackage{lmodern}
\usepackage{amsmath,amssymb,mathtools,mathrsfs}
\usepackage{enumitem,needspace}
\usepackage[dvipsnames]{xcolor}
\usepackage[colorlinks=true,linkcolor=orange,
 citecolor=orange,urlcolor=red!50!black]{hyperref}
\allowdisplaybreaks[1]
\numberwithin{equation}{section}
\newtheorem{theorem}{Theorem}[section]
\newtheorem{proposition}[theorem]{Proposition}
\newtheorem{lemma}[theorem]{Lemma}
\newtheorem{corollary}[theorem]{Corollary}
\theoremstyle{definition}

\theoremstyle{remark}

\setlist[enumerate,1]{label=\textup{(\arabic*)},leftmargin=2em,
 itemsep=.2\baselineskip}
\DeclareMathOperator{\tr}{tr}
\DeclareMathOperator{\Ric}{Ric}
\DeclareMathOperator{\Ent}{Ent}
\DeclareMathOperator{\PSH}{PSH}
\DeclareMathOperator{\MA}{MA}
\DeclareMathOperator{\Lip}{Lip}
\DeclareMathOperator{\diam}{diam}
\definecolor{revisiongreen}{RGB}{0,128,0}
\newcommand{\changed}[1]{{#1}}
\newcommand{\mathchanged}[1]{{#1}}
\newcommand{\ddc}{dd^c}
\newcommand{\R}{\mathbb R}
\newcommand{\C}{\mathbb C}
\newcommand{\cM}{\mathcal{M}}
\newcommand{\hw}{\mathchanged{\tilde\omega}}
\newcommand{\pcan}{\varphi_{\rm can}}
\newcommand{\ocan}{\omega_{\rm can}}
\newcommand{\oep}{\omega_{\varepsilon}}
\newcommand{\osrf}{\omega_{\rm SRF}}
\newcommand{\tom}{\widetilde{\omega}}
\newcommand{\pr}{\operatorname{pr}}
\newcommand{\dcan}{d_{\rm can}}

\title[Canonical limits of cscK metrics]
{Gromov--Hausdorff convergence of constant scalar curvature
K\"ahler metrics near the canonical class}
\author{Frederick Tsz-Ho Fong}
\author{Junbang Liu}

\address{Department of Mathematics, The Hong Kong University of Science and
Technology, Clear Water Bay, Kowloon, Hong Kong}
\email{frederick.fong@ust.hk}
\email{junbangliu@ust.hk}
\date{}

\hypersetup{pdftitle={Gromov-Hausdorff convergence of constant scalar curvature
Kahler metrics near the canonical class},pdfauthor={Junbang Liu}}
\begin{document}

\begin{abstract}
Let $X$ be a compact K\"ahler manifold with semi-ample canonical bundle. For any K\"ahler class $[\mathchanged{\omega}]$, Jian--Shi--Song \cite{JSS18} proved that there exists a unique cscK metric in the class $2\pi c_1(K_X)+\varepsilon[\mathchanged{\omega}]$ when $\varepsilon\in (0,\varepsilon_0]$, based on the properness criterion of Chen--Cheng \cite{CC17}. 
We prove that the cscK metrics in
$2\pi c_1(K_X)+\varepsilon[\mathchanged{\omega}]$ converge, as
$\varepsilon\downarrow0$, in the Gromov--Hausdorff topology to the
metric completion of the twisted K\"ahler-Einstein metric on the
canonical model of $X$ as conjectured by Jian--Shi--Song in \cite{JSS18}. Based on the technique developed in this paper, we can also give a new proof the theorem of Lee--Tosatti--Zhang \cite{LTZ26} on the Gromov-Hausdorff convergence of the normalized K\"ahler--Ricci flow on $X$.
\end{abstract}
\maketitle
\setcounter{tocdepth}{1}
{\hypersetup{linkcolor=black}\tableofcontents}

\section{Introduction}
Let $X$ be a connected compact K\"ahler manifold of complex dimension
$n$ with semi-ample canonical bundle $K_X$. Fix a K\"ahler class $[\mathchanged{\omega}]$. \changed{Jian--Shi--Song} \cite[Theorem 1.1]{JSS18} proved that $-2\pi c_1(X)+\varepsilon[\mathchanged{\omega}]$ contains a unique cscK metric for all sufficiently small $\varepsilon>0$ based on the properness criterion of \changed{Chen--Cheng} \cite{CC17}. They conjectured that these metrics converge to
the canonical metric on the pluricanonical base, both locally away
from singular fibers and globally in the Gromov--Hausdorff topology
\cite[Conjecture~1.1]{JSS18}. 

There are two questions in the degeneration: local convergence of the metrics on the regular part of the fibration, and identification of the global metric limit. In the big and nef case, W. Liu
\cite[Theorem~1.3]{L21} established smooth convergence outside an
effective divisor to the singular negative K\"ahler--Einstein metric.
For an arbitrary semi-ample canonical bundle, Guo--Jian--Shi--Song
\cite[Theorems~1.1 and~2.1]{GJSS24} obtained uniform bounds for
potentials, normalized volume densities and diameters, together with
lower bounds for normalized volumes of metric balls. These imply
Gromov--Hausdorff precompactness. The goal is to identify every limit with
the canonical metric completion.

The corresponding degeneration has also been studied through the
K\"ahler--Ricci flow. Tosatti--Weinkove--Yang
\cite{TWY18} established local convergence to the canonical metric
and convergence of the rescaled fiber metrics. Related geometric
estimates and metric convergence results were obtained in
\cite{FGS19,STZ19,S25}. Lee--Tosatti--Zhang \cite[Theorem~1.2]{LTZ26}
proved the global Gromov--Hausdorff convergence of the normalized K\"ahler--Ricci
flow. 

\changed{We now describe the setting. Since $K_X$ is semi-ample,} $\ell K_X$ is base-point free for some $\ell\ge 1$, and the linear system $|\ell K_X|$ for $\ell$ sufficiently divisible gives a fiber space (the Iitaka fibration of $X$):
$f:X\to Y\subset \C P^N$ with $\ell K_X=f^*(\mathcal{O}_{\C P^N}(1)|_Y)$, onto a normal projective variety $Y$.  Put $m=\dim_\C Y=\kappa(X)$, which is the Kodaira dimension. The generic fibers of $f$ are \changed{Calabi--Yau} $d$-folds with $d=n-m$. For $m>0$, let
\[
 Y^\circ=Y_{\rm reg}\setminus\{\text{discriminant locus of }f\},
 \qquad X^\circ=f^{-1}(Y^\circ).
\]
\changed{Then $f:X^\circ\to Y^\circ$ is a proper holomorphic submersion} with connected $d$-dimensional \changed{Calabi--Yau} fibers. 
The twisted K\"ahler-Einstein metric $\omega_{\rm can}$ on $Y^\circ$
satisfies
\[
 \Ric(\omega_{\rm can})=-\omega_{\rm can}+\omega_{\rm WP},
\]
where $\omega_{\rm WP}$ is the Weil--Petersson form of the smooth
Calabi--Yau fibration. Let $\dcan$ be the distance induced by $\ocan$ on $Y^\circ$ and $\nu_{\rm can}$ be the normalized volume form for the canonical metric. Its metric completion $(\overline{Y^\circ},\dcan, \nu_{\rm can})$ is a compact $RCD(-1,2m)$ space and
\changed{homeomorphic to $Y$, as shown by Sz\'ekelyhidi} \cite{S25RCD}. The main result in this paper is
\begin{theorem}\label{thm:main}
Let $X$ be a connected compact K\"ahler manifold with semi-ample
canonical bundle, and let $[\mathchanged{\omega}]$  be any K\"ahler class.
For all sufficiently small $\varepsilon>0$, let
$\omega_\varepsilon\in 2\pi c_1(K_X)+\varepsilon [\mathchanged{\omega}] $ be the unique cscK metric. If $m=0$, the metrics converge to a point in the Gromov-Hausdorff topology. If $m>0$, then:\begin{enumerate}
    \item For any compact $K\Subset Y^\circ$, 
    \begin{equation}
        \|\omega_\varepsilon-(f^*\ocan+\varepsilon\omega_{\rm SRF})\|_{f^*\ocan+\varepsilon\mathchanged{\omega}}\to 0 \quad \text{uniformly on }f^{-1}(K).
    \end{equation}
    \item Let $d_{\oep}$ be \changed{the distance function induced by the metric} $\oep$ on $X$, \begin{equation}\label{eq:gh}
(X,d_{\omega_\varepsilon})
 \longrightarrow (\overline{Y^\circ},d_{\rm can})
 \qquad\text{in the Gromov--Hausdorff topology}.
 \end{equation}
\end{enumerate}
\end{theorem}
The meaning of $\omega_{\rm SRF}$ will be explained in the next section. For $m=0$, choose $\chi=0$. The cscK equations give \[
\underline{R}_\varepsilon=0,\qquad \Delta_{\oep}F_\varepsilon=0,\qquad \Ric(\oep)=0.
\]
If $\omega_{\rm RF}\in [\mathchanged{\omega}]$ is the Ricci-flat metric, uniqueness gives \[
\oep=\varepsilon\omega_{\rm RF}, \qquad \operatorname{diam}(X,\oep)=\sqrt{\varepsilon}\operatorname{diam}(X,\omega_{\rm RF})\to 0.
\] Hence, we assume $m>0$ in the rest of the paper. 

\textbf{Main ideas of the proof}. The starting point is Guo--Jian--Shi--Song's \changed{uniform $L^\infty$ estimates for the potential and the volume density}. \changed{They established uniform noncollapsing estimates and a uniform Sobolev inequality} for this family of cscK metrics. \changed{These estimates enable us to establish a $C^2$-estimate in the base direction and a gradient estimate for the logarithm of the volume density}:\[
\tr_{\oep}\chi\leq C, \qquad |\nabla F_\varepsilon|_{\oep}\leq C.
\] The method is inspired by the \changed{groundbreaking work of Chen--Cheng} \cite{CC17}, but adapted to the fibration structure of the canonical model here.  The next step is to show that the subsequential weak limit of the volume density $F_\varepsilon$ is exactly the \changed{pullback} of the canonical one. The strategy for this step \changed{is to use the minimizing property of the Mabuchi energy} for cscK metrics and a variational characterization of the Monge-Amp\`ere measures. Next, after identifying the weak limit, the uniform Sobolev inequality enables us to improve the weak convergence to $L^\infty$-convergence away from the singular fibers. This is done by \changed{local Moser iteration, which is available thanks to the uniform Sobolev inequality obtained by Guo--Phong--Song--Sturm} \cite{GPSS23}.

After identifying the limit of the \changed{normalized} volume form, one can show an integral estimate for the difference of $\omega_\varepsilon$ and the metric in the same class from the normalized \changed{K\"ahler--Ricci flow}. With local Moser iteration, one can improve the \changed{previous bound $\tr_{\oep}\chi\leq C$ to} the sharp \changed{Schwarz-type} estimate $\omega_{\varepsilon}\geq (1+o(1))f^*\ocan$.   

With the sharp \changed{Schwarz-type} estimate, we are able to show that the rescaled fiber potential of $\oep$ converges uniformly to the potential of $\omega_{\rm SRF}$. Then a blow-up argument, together with the crucial $W^{2,p}$-estimate of \changed{Cheng--Xu} \cite{CX23}, \changed{and the local estimates for cscK metrics of Chen--Cheng} \cite{CC17} proves the metric convergence away from singular fibers.

After obtaining the metric convergence away from singular fibers, the full Gromov-Hausdorff convergence can be reduced to \changed{a one-sided estimate of the distance functions (by} the work of Lee--Tosatti--Zhang \cite{LTZ26}): \[
\sup_{x,x'\in X}
\big(d_{\rm can}(f(x),f(x'))-d_{\omega_\varepsilon}(x,x')\big)_+
\longrightarrow0 .
\]
\changed{In the setting of the K\"ahler--Ricci flow, an important ingredient} in Lee--Tosatti--Zhang's proof is \changed{Perelman's} monotonicity of the reduced volume. Here \changed{we develop another method based on purely analytic propositions} about Lipschitz functions. We show that the above estimate on the distance functions can be reduced to the following Lipschitz lift proposition: for any $v\in \Lip_{\dcan}(Y)$ with $\Lip(v)\leq 1$, and $\delta>0$, there exists $v_\varepsilon\in C^\infty(X)$ such that $\limsup \|v_\varepsilon-f^*v\|_{L^\infty}\leq \delta$ and $\Lip_{d_{\oep}}(v_\varepsilon)\leq1+\delta.$ \changed{The RCD structure of $Y$ and volume noncollapsing play an important role here. Finally,} we show that this method \changed{can also be applied to the normalized K\"ahler--Ricci flow to give an independent proof of the theorem} of \changed{Lee--Tosatti--Zhang}. 

\textbf{Acknowledgements}. The second author is sincerely grateful to Prof. Xiuxiong Chen for his continued encouragement and \changed{inspiring discussions on cscK metrics}. He would also like to thank Hanbing Fang for \changed{helpful discussions on the convergence problem} for the family of cscK metrics. \changed{This research is partially supported} by the Hong Kong General Research Fund \#16305625 of the Hong Kong Research Grants Council.

\textbf{Declaration on the use of AI}. The key $C^2$-estimate in the base direction and gradient \changed{bound for the logarithm of the volume density} are due to the authors. \changed{We then decomposed the proof into several steps and lemmas and asked ChatGPT 6 to produce counterexamples or proofs in special cases. The feedback inspired ideas that helped us adjust our approach. We arrived at the appropriate lemmas after several rounds of intensive discussion. The entire proof was written by the authors.} 
\section{Basic facts}
We give a preliminary discussion on the family of cscK metrics, known results for uniform estimates along this family, and \changed{known results for the normalized K\"ahler--Ricci flow}. 

We use the conventions \[
\ddc =\sqrt{-1}\partial\bar\partial, \qquad \Delta_{\omega}u=\tr_{\omega}\ddc u, \qquad |\nabla u|_{\omega}=|\partial u|_{\omega}=\frac{1}{\sqrt{2}}|\nabla u|_{g_{\omega}},
\]
where $g_{\omega}$ is the Riemannian metric \changed{corresponding to the K\"ahler form} $\omega.$ The distance and \changed{Lipschitz} constants are defined using the Riemannian metric $g_\omega$, and hence $\Lip_{d_\omega}(u)=\sqrt{2}\||\nabla u|_{\omega}\|_{L^\infty}$. 

Choose a sufficiently large $\ell$ such that $\ell K_X=f^*\mathcal O_Y(1)$ and the pluricanonical map has connected fibers. Set\[
\chi_Y=\ell^{-1}\omega_{FS}|_Y, \qquad \chi=f^*\chi_Y,\qquad \chi_\varepsilon=\chi+\varepsilon\mathchanged{\omega}.
\]
Here $\mathchanged{\omega}$ is a fixed smooth K\"ahler form \changed{in the K\"ahler class} $[\mathchanged{\omega}]$. Let $m$ be the Kodaira dimension, which is equal to the dimension of $Y$ here. Let $d=n-m$, $V_\varepsilon=\int_X\chi_\varepsilon^n, d\mu_\varepsilon=V_\varepsilon^{-1}\omega_\varepsilon^n,$ where $\omega_\varepsilon$ is the unique cscK metric in $[\chi_\varepsilon]$.
We fix a smooth positive volume form $\Omega$ on $X$ with $\int_X\Omega=1$, and $\Ric(\Omega)=-\chi$. 
Let $D\subsetneq Y$ be the singular locus together with the critical
values of $f$. Then
\[Y^\circ=Y\setminus D,\qquad X^\circ=f^{-1}(Y^\circ),\qquad
f:X^\circ\longrightarrow Y^\circ \quad\text{is a proper submersion with connected fibers}.
\]
For $m>0$, \changed{Song--Tian's} canonical metric \cite{ST07,ST12} $\omega_{\rm can}=\chi_Y+\ddc \varphi_{\rm can}$ is the unique solution of the twisted K\"ahler-Einstein equation\[
\Ric(\ocan)=-\ocan+\omega_{WP},
\]where $\omega_{WP}\geq 0$ is the \changed{Weil--Petersson} form of the \changed{Calabi--Yau} fibration $X^\circ\to Y^\circ$. The potential $\varphi_{\rm can}$ satisfies $\varphi_{\rm can}\in \PSH(Y,\chi_Y)\cap L^\infty(Y)\cap C^\infty(Y^\circ)$, and $\ocan$ extends as a positive current on $Y$. 

Set \[
\mathchanged{\omega_y}=\mathchanged{\omega}|_{X_y}, \quad V_F=\int_{X_y}\mathchanged{\omega_y}^d, \quad V_Y=\int_Y\chi_Y^m.
\]
When $d=0$, $V_F:=1$. For $d>0$, 
on $X^\circ$, we have the semi-Ricci-flat form\[
 \omega_{\rm SRF}=\mathchanged{\omega}+dd^c\rho,\qquad
 \int_{X_y}\rho\mathchanged{\,\omega_y^d}=0,\qquad
 \omega_{{\rm SRF},y}:=\omega_{\rm SRF}|_{X_y}>0,
 \qquad\Ric(\omega_{{\rm SRF},y})=0.
\]
The form $\omega_{\rm SRF}$ is smooth but need not be positive on
$X^\circ$. For $d=0$, set $\rho=0$ and $\omega_{\rm SRF}=\mathchanged{\omega}$.  \changed{We have} \begin{equation}\label{eq:cv}
 e^{f^*\varphi_{\rm can}}\Omega
 =\frac{(f^*\omega_{\rm can})^m\wedge\omega_{\rm SRF}^d}{V_FV_Y}.
\end{equation}
and \begin{equation}
(\chi_Y+\ddc\varphi_{\rm can})^m=V_Ye^{\varphi_{\rm can}}f_*\Omega
\end{equation}
Let $\dcan$ be the distance on $Y^\circ$ induced by the canonical metric $\ocan$.
The canonical measure $\nu_{\rm can}$ is defined as $
\nu_{\rm can}:=\frac{\omega_{\rm can}^m}{V_Y}=e^{\varphi_{\rm can}}f_*\Omega.$

Recall first the following theorem of Sz\'ekelyhidi \changed{on the RCD structure of} the canonical model. 
\begin{theorem}[{Sz\'ekelyhidi \cite[Theorem~15]{S25}}]\label{thm:completion}
For the semi-ample canonical fibration with $m>0$, the completion
$\overline{Y^\circ}=\overline{(Y^\circ,d_{\rm can})}$ satisfies:
\begin{enumerate}
\item\label{item:thm-completion-1} $\overline{Y^\circ}$ is compact and canonically homeomorphic to $Y$.
\item\label{item:thm-completion-2} $(Y,\dcan,\nu_{\rm can})$ is a noncollapsed
$\operatorname{RCD}(-1,2m)$ space. 
\item\label{item:thm-completion-3}
$\dim_{\mathcal H}(Y\setminus Y^\circ)\leq2m-2$.
\end{enumerate}
\end{theorem}

Now we recall the equations of cscK metrics.
The cscK equation can be expressed as \[\begin{cases}
    \omega_\varepsilon^n=V_\varepsilon e^{F_\varepsilon}\Omega\\
    \Delta_{\omega_\varepsilon}F_\varepsilon=-\underline{R}_\varepsilon-\tr_{\omega_\varepsilon}\chi
\end{cases}
\]
We also set $\omega_\varepsilon=\chi_\varepsilon+\ddc \varphi_\varepsilon$ with $\sup_X\varphi_\varepsilon=0$. \changed{The average scalar curvature} can be computed as \[
\underline R_\varepsilon
 =-n\frac{A\cdot(A+\varepsilon B)^{n-1}}{(A+\varepsilon B)^n}, \quad \text{ where }A=2\pi c_1(K_X), B=[\mathchanged{\omega}].
\]
\changed{Put $a_j=\int_X\chi^j\wedge\mathchanged{\omega}^{n-j}$. Then} \begin{enumerate}
    \item $a_m>0$, $\chi^{m+1}=0$, and $V_\varepsilon=\sum_{j=0}^m\binom nj a_j\varepsilon^{n-j}.$
    \item If $m>0$, then
\[
\begin{aligned}
 V_\varepsilon
 &=\binom nm a_m\varepsilon^d
   +\binom n{m-1}a_{m-1}\varepsilon^{d+1}
   +O(\varepsilon^{d+2}),\\
 \underline R_\varepsilon
 &=-m+\frac{m a_{m-1}}{(d+1)a_m}\varepsilon+O(\varepsilon^2).
\end{aligned}
\]
\item\label{item:zero-dimensional-asymptotics} If $m=0$, then
$V_\varepsilon=a_0\varepsilon^n$ and $\underline R_\varepsilon=0$.
\end{enumerate}

We also recall the uniform estimates of Guo--Jian--Shi--Song. 
\begin{theorem}[{Guo--Jian--Shi--Song \cite[Theorems~1.1 and~2.1]{GJSS24}}]\label{thm:bounds}
There are constants $C,\varepsilon_0,a_{\rm vol}>0$ such that
for $0<\varepsilon<\varepsilon_0$:
\begin{enumerate}
\item\label{item:thm-bounds-1}
$\|\varphi_\varepsilon\|_\infty+\|F_\varepsilon\|_\infty\leq C$,
and hence $C^{-1}\Omega\leq\mathchanged{d\mu_\varepsilon}\leq C\Omega$.
\item\label{item:thm-bounds-2} $\diam(X,\omega_\varepsilon)\leq C$.
\item\label{item:thm-bounds-3}
$\mu_\varepsilon(B_{\omega_\varepsilon}(x,r))\geq C^{-1}r^{a_{\rm vol}}$
for $x\in X$, $0<r<1$. Consequently the family is GH precompact.
\end{enumerate}
\end{theorem}

\changed{To prove higher-order estimates}, we also need the uniform Sobolev inequality obtained by Guo--Phong--Song--Sturm. 
\begingroup\color{black}
\begin{theorem}[{Guo--Phong--Song--Sturm \cite[Theorem~2.1]{GPSS23}}]\label{thm:sobolev}
Let $\mathchanged{(X,\omega)}$ be a fixed compact K\"ahler manifold. For a K\"ahler form $\mathchanged{\eta}$, set
\[
\color{black}
 V_\eta=\int_X\eta^n,\qquad d\mu_\eta=\frac{\eta^n}{V_\eta},\qquad
 I_\eta=[\eta]\cdot[\omega]^{n-1},\qquad V_\omega=\int_X\omega^n.
\]
Fix $\mathchanged{p>n}$ and constants $\mathchanged{K,c>0}$. Suppose that
\[
\color{black}
 \int_X\left|\log\frac{\eta^n/V_\eta}{\omega^n/V_\omega}\right|^p d\mu_\eta\leq K,
 \qquad \frac{\eta^n/V_\eta}{\omega^n/V_\omega}\geq c.
\]
Then there exist constants $\mathchanged{q=q(n,p)>1}$ and $\mathchanged{C_S>0}$, depending only on the fixed background, $\mathchanged{n,p,K,c}$, and $\mathchanged{q}$, such that, for every $\mathchanged{u\in W^{1,2}(X,\eta)}$,
\[
\color{black}
 \left(\int_X|u-\bar u_\eta|^{2q}d\mu_\eta\right)^{1/q}
 \leq C_S I_\eta\int_X|\nabla u|_\eta^2d\mu_\eta,
 \qquad \bar u_\eta=\int_Xu\,d\mu_\eta.
\]
\end{theorem}
This is the positive-constant lower-density special case of the cited theorem. For $\mathchanged{\eta=\omega_\varepsilon}$, Theorem~\ref{thm:bounds} and the smooth positivity of $\mathchanged{\Omega}$ give both hypotheses uniformly, while $\mathchanged{I_{\omega_\varepsilon}}$ is uniformly bounded. Thus the Sobolev constant for the cscK family is uniform.
\endgroup

We will also need the convergence results of normalized \changed{K\"ahler--Ricci flow}, \[
\begin{cases}
    \partial_t\mathchanged{\tilde\omega}(t)=-\Ric(\mathchanged{\tilde\omega}(t))-\mathchanged{\tilde\omega}(t),\\
    \mathchanged{\tilde\omega}(0)=\mathchanged{\omega}.
\end{cases}
\]
It can be written \changed{as a scalar equation:}\[
\begin{cases}
    \partial_t\mathchanged{\tilde\varphi}=\log\frac{e^{dt}(e^{-t}\mathchanged{\omega}+(1-e^{-t})\chi+\ddc\mathchanged{\tilde\varphi})^n}{\binom{n}{d}V_FV_Y\Omega}-\mathchanged{\tilde\varphi},\\
    \mathchanged{\tilde\varphi}(0)=0.
\end{cases}
\]
We recall the following facts for the above normalized \changed{K\"ahler--Ricci flow}, see \cite{TWY18} and \changed{the references therein}. 
\begin{theorem}
\label{thm:NKRF-convergence}
\leavevmode
\begin{enumerate}
    \item $|\mathchanged{\tilde\varphi}(t)|+|\partial_t\mathchanged{\tilde\varphi}(t)|\leq C$ and $C^{-1}e^{-dt}\mathchanged{\omega}^n\leq (\mathchanged{\tilde\omega}(t))^n\leq Ce^{-dt}\mathchanged{\omega}^n$. \changed{See Song--Tian} \cite[Lemma~2.1 and Proposition~2.1]{ST16}\changed{; the volume bounds follow from the scalar flow equation and these potential bounds.}
    \item $\mathchanged{\tilde\varphi}(t)\to f^*\varphi_{\rm can}$ and $\partial_t\mathchanged{\tilde\varphi}\to 0$ uniformly on every compact set of $X^\circ,$ and in $L^p(X,\Omega)$ for any $1\le p<\infty$. \changed{For convergence of the potential, see Song--Tian} \cite[Proposition~5.4]{ST12}\changed{; for convergence of the time derivative on the regular locus, see Tosatti--Weinkove--Yang} \cite[Lemma~3.1(ii), (iv)]{TWY18}\changed{. The finite-}$L^p$\changed{ convergence follows from part (1) and dominated convergence.}
    \item $\|\mathchanged{\tilde\omega}(t)-(e^{-t}\omega_{\rm SRF}+(1-e^{-t})f^*\omega_{\rm can})\|_{\mathchanged{\tilde\omega}(t)}\to 0$ on every compact subset of $X^\circ$. \changed{See Tosatti--Weinkove--Yang} \cite[Section~3.3, proof of Theorem~1.2(i), immediately after equation~(3.22)]{TWY18}\changed{, where this relative-norm convergence is established using Lemma~2.6.}
\end{enumerate}
\end{theorem}

We set \[t(\varepsilon)=\log\frac{1+\varepsilon}{\varepsilon},\qquad
\mathchanged{\tilde\omega}_\varepsilon=(1+\varepsilon)\mathchanged{\tilde\omega}(t(\varepsilon)),
\qquad\psi_\varepsilon=(1+\varepsilon)\mathchanged{\tilde\varphi}(t(\varepsilon)).
\]
\changed{This time reparametrization and rescaling adjust the K\"ahler--Ricci flow metric to lie in the same K\"ahler class as the cscK metric} $\omega_\varepsilon$\changed{. Indeed,} $\mathchanged{\tilde\omega}_\varepsilon=\chi+\varepsilon\mathchanged{\omega}+\ddc \psi_\varepsilon\in 2\pi c_1(K_X)+\varepsilon [\mathchanged{\omega}]$. The above convergence and estimates give \[
\psi_\varepsilon\to f^*\varphi_{\rm can} \text{ in }L^p(X,\Omega), \qquad \|\psi_\varepsilon\|_{L^\infty}\leq C,
\] and \[
\frac{\mathchanged{\tilde\omega}_\varepsilon^n}{V_\varepsilon}
 =\frac{\binom nd V_FV_Y(1+\varepsilon)^m\varepsilon^d}{V_\varepsilon}
                e^{(\mathchanged{\tilde\varphi}+\mathchanged{\dot{\tilde\varphi}})(t(\varepsilon))}\Omega
 \longrightarrow e^{f^*\varphi_{\rm can}}\Omega
       \quad\text{in total variation}.
\]

\section{Uniform second-order estimate in base direction for the cscK metrics}
In this section, we show that the second-order estimate in the base direction for the path of cscK metrics. \changed{This is analogous to the parabolic Schwarz lemma of Song--Tian}~\cite{ST07}\changed{, whose proof uses the parabolic maximum principle.} The proof is inspired by the work of Chen-Cheng~\cite{CC17}, and adapted to our fibration structure over the canonical model here.  

Fix $\varepsilon$, set
\[
 \varphi=\varphi_\varepsilon,\quad
 \omega_\varphi=\omega_\varepsilon,\quad F=F_\varepsilon,\quad
 \underline R=\underline R_\varepsilon,\quad
 \Delta_\varphi=\Delta_{\omega_\varphi},\quad
 d\mu=\frac{\omega_\varphi^n}{\int_X\omega_\varphi^n}.
\]
We write $\tr_\varphi\chi=\tr_{\omega_\varphi}\chi$ and use
$\nabla_\varphi$ for covariant differentiation with respect to
$g_\varphi$.  In this section we prove the $C^2$-estimate in the base direction in the following order:
\[
 \|\tr_\varphi\chi\|_{L^p(\mu)}\leq C_p\quad(p<\infty)
 \ \Longrightarrow\ 
 \|\nabla_\varphi F\|_\infty\leq C
 \ \Longrightarrow\ 
 \|\tr_\varphi\chi\|_\infty\leq C.
\]
At a prescribed point, choose holomorphic normal coordinates for
$g_\varphi$ and diagonalize $\chi$:
\begin{equation}\label{eq:estimate-coordinates}
g_{\varphi,i\bar j}=\delta_{ij},\qquad
 \partial_k g_{\varphi,i\bar j}=0,\qquad
 \chi_{i\bar j}=\lambda_i\delta_{ij},\qquad \lambda_i\geq0.
\end{equation}

\changed{The following lemma provides the first step above: uniform bounds for the trace of the pullback metric in every finite} $\mathchanged{L^p}$\changed{ space. We formulate it for a holomorphic map into a fixed K\"ahler target. In our application, we take} $\mathchanged{Z=\C P^N}$\changed{, equipped with} $\mathchanged{\omega_Z=\ell^{-1}\omega_{FS}}$\changed{, and} $\mathchanged{\Phi=\iota\circ f}$\changed{, where} $\mathchanged{\iota:Y\hookrightarrow\C P^N}$\changed{ is the inclusion. Thus} $\mathchanged{\Phi}$\changed{ is the pluricanonical map supplied by the semi-ampleness of} $K_X$\changed{, and its pullback metric is precisely} $\chi$\changed{.}

\begin{lemma}\label{lem:trace-lp}
Let $\Phi:X\to Z$ be a holomorphic map from a compact complex manifold \changed{$X$}
to a fixed compact K\"ahler manifold $(Z,\omega_Z)$, and put
$\chi=\Phi^*\omega_Z$. Suppose that $\omega_\varphi$, $F$, and a constant
$\underline R$ satisfy
\begin{equation}\label{eq:ht1}
 \Ric(\omega_\varphi)=-\chi-dd^cF,\qquad
 \Delta_\varphi F=-\underline R-\tr_\varphi\chi.
\end{equation}
If $|\underline R|+\operatorname{osc}_X F\leq C_0$, then, for every
finite $p\geq1$,
\begin{equation}\label{eq:ht2}
\int_X(\tr_\varphi\chi)^p\,d\mu\leq C_p.
\end{equation}
The constant \changed{$C_p$} depends only on $p,C_0$ and an upper bound for the
holomorphic bisectional curvature of $\omega_Z$.
\end{lemma}
\begin{proof}
Integration of \eqref{eq:ht1} gives
$-\underline R=\int_X\tr_\varphi\chi\,d\mu\geq0$.
\changed{Choose source coordinates satisfying} \eqref{eq:estimate-coordinates} \changed{at} $x_0$\changed{, and normal target coordinates at} $\Phi(x_0)$, and write
$\omega_Z=\sqrt{-1}h_{\alpha\bar\beta}dw^\alpha\wedge d\bar w^\beta$.
Then
\[
 \chi_{i\bar j}=\sum_{\alpha,\beta}
 h_{\alpha\bar\beta}(\Phi)\Phi_i^\alpha\overline{\Phi_j^\beta},
\qquad \tr_\varphi\chi=\sum_{i,\alpha}|\Phi_i^\alpha|^2 =\sum_i\lambda_i.
\]
With $R_{i\bar j k\bar l}=-\partial_k\partial_{\bar l}
g_{\varphi,i\bar j}$ at the normal point, differentiation gives
\begin{align}
 (\tr_\varphi\chi)_k&=\sum_{i,\alpha}\Phi_{ik}^\alpha\overline{\Phi_i^\alpha},
       \label{eq:trace-first}\\
 \Delta_\varphi\tr_\varphi\chi
 &=\sum_{i,k,\alpha}|\Phi_{ik}^\alpha|^2
   +\sum_i\lambda_i\Ric(\omega_\varphi)_{i\bar i}-\sum_{i,k,\alpha,\beta,\gamma,\delta}R^{\omega_Z}_{\alpha\bar\beta\gamma\bar\delta}\Phi_i^\alpha\overline{\Phi_i^\beta}\Phi_k^\gamma\overline{\Phi_k^\delta}.
       \label{eq:map-bochner-coordinates}
\end{align}
By Cauchy--Schwarz inequality,
\[
|\nabla_\varphi\tr_\varphi\chi|_\varphi^2
\leq(\tr_\varphi\chi)\sum_{i,k,\alpha}|\Phi_{ik}^\alpha|^2.\]
By \eqref{eq:ht1},
\[\sum_i\lambda_i\Ric(\omega_\varphi)_{i\bar i}=-\sum_i\lambda_i^2-\sum_i\lambda_iF_{i\bar i}.
\]
Choose $K\geq1$ so that the bisectional curvature of $\omega_Z$ is at most
$K-1$. Raising both indices of $\chi$ with $g_\varphi$, we obtain
\begin{equation}\label{eq:ht3}
 \Delta_\varphi\tr_\varphi\chi
 \geq\sum_{i,k,\alpha}|\Phi_{ik}^\alpha|^2
       -K(\tr_\varphi\chi)^2-\sum_{i,j}\chi^{i\bar j}F_{i\bar j}.
\end{equation}
Here $\chi^{i\bar j}=g_\varphi^{i\bar b}\chi_{a\bar b}
g_\varphi^{a\bar j}$, so $\chi^{i\bar j}=\lambda_i\delta_{ij}$ at
the normal point.

For $0<\delta\leq1$, put $v=\tr_\varphi\chi+\delta$. Since $v>0$,
\begin{equation}\label{eq:ht4}
 \Delta_\varphi v\geq\frac{|\nabla_\varphi v|_\varphi^2}{v}
       -Kv^2-\sum_{i,j}\chi^{i\bar j}F_{i\bar j},\qquad
 \Delta_\varphi F=-\underline R+\delta-v.
\end{equation}
The closedness of $\chi$ gives, at the normal point,
\begin{equation}\label{eq:ht5}
\sum_i\nabla_{\varphi,i}\chi^{i\bar{j}}=\sum_i\nabla_{\varphi,i}\chi_{j\bar i}=\sum_i\nabla_{\varphi,j}\chi_{i\bar i}=v_j.
\end{equation}

Fix $p\geq1$, and choose $b\in \R$ later. Testing \eqref{eq:ht4} with
$e^{bF}v^{p-1}$,
\begin{align*}
 \int_Xe^{bF}v^{p-1}\Delta_\varphi v\,d\mu
 &=-\int_Xe^{bF}v^p\sum_i\left[
 (p-1)\left|\frac{v_i}{v}\right|^2+b\operatorname{Re}\left(\frac{v_i}{v}F_{\bar i}\right)\right]d\mu,\\
 \int_Xe^{bF}v^{p-2}|\nabla_\varphi v|_\varphi^2\,d\mu
&=\int_Xe^{bF}v^p\sum_i\left|\frac{v_i}{v}\right|^2d\mu,\\
 -\int_Xe^{bF}v^{p-1}\chi^{i\bar j}F_{i\bar j}\,d\mu
 &=\operatorname{Re}\int_X
 \nabla_{\varphi,i}(e^{bF}v^{p-1}\chi^{i\bar j})F_{\bar j}\,d\mu.
\end{align*}
Repeated indices in the last line are summed. Expanding its derivative
and using \eqref{eq:ht5} gives
\[
\begin{aligned}
 &\operatorname{Re}\sum_{i,j}
 \nabla_{\varphi,i}(e^{bF}v^{p-1}\chi^{i\bar j})F_{\bar j}\\
 &\quad=e^{bF}v^p\sum_i\left[
 b\frac{\lambda_i}{v}|F_i|^2
 +(p-1)\frac{\lambda_i}{v}
              \operatorname{Re}\left(\frac{v_i}{v}F_{\bar i}\right)
 +\operatorname{Re}\left(\frac{v_i}{v}F_{\bar i}\right)\right].
\end{aligned}
\]
Consequently,
\begin{equation}\label{eq:ht6}
\begin{split}
 &\int_Xe^{bF}v^p\sum_i\left[
 p\left|\frac{v_i}{v}\right|^2
 +\left((p-1)\frac{\lambda_i}{v}+b+1\right)
             \operatorname{Re}\left(\frac{v_i}{v}F_{\bar i}\right)
 +b\frac{\lambda_i}{v}|F_i|^2\right]d\mu\\
 &\hspace{45mm}\leq K\int_Xe^{bF}v^{p+1}\,d\mu.
\end{split}
\end{equation}
Testing the equation $\Delta_\varphi F=-\underline{R}+\delta-v$ with $e^{bF}v^p$ similarly gives
\begin{equation}\label{eq:ht7}
\begin{split}
 &\int_Xe^{bF}v^p\sum_i\left[
 p\operatorname{Re}\left(\frac{v_i}{v}F_{\bar i}\right)+b|F_i|^2
 \right]d\mu\\
 &\qquad=\int_Xe^{bF}v^{p+1}\,d\mu
       -(-\underline R+\delta)\int_Xe^{bF}v^p\,d\mu.
\end{split}
\end{equation}

Subtract $\Lambda$ times \eqref{eq:ht7} from \eqref{eq:ht6}:
\begin{equation}\label{eq:trace-combination}
\begin{split}
 \int_Xe^{bF}v^p\mathcal Q\,d\mu
 +(\Lambda-K)\int_Xe^{bF}v^{p+1}\,d\mu
 \leq\Lambda(-\underline R+\delta)\int_Xe^{bF}v^p\,d\mu,
\end{split}
\end{equation}
where
\[
\begin{split}
 \mathcal Q=\sum_i\bigg[
 &p\left|\frac{v_i}{v}\right|^2
 +\left((p-1)\frac{\lambda_i}{v}+b+1-\Lambda p\right)
        \operatorname{Re}\left(\frac{v_i}{v}F_{\bar i}\right)\\
 &+b\left(\frac{\lambda_i}{v}-\Lambda\right)|F_i|^2\bigg].
\end{split}
\]
We seek $\Lambda>K$ and $\mathcal Q\geq0$. At the endpoint
$\lambda_i/v=1$, requiring the summand to be
$p|v_i/v-\alpha F_i|^2$ gives
\[
 \begin{cases}
 p+b-\Lambda p=-2p\alpha,\\
 b(1-\Lambda)=p\alpha^2.
 \end{cases}
 \quad\Longrightarrow\quad
 (\Lambda-1-\alpha)^2=0,
 \qquad \Lambda=\alpha+1,\quad b=-p\alpha.
\]
For these choices, set $u=e^{-\alpha F}v$. Then
\[
 (\log u)_i=\frac{v_i}{v}-\alpha F_i,
 \qquad 0\leq\frac{\lambda_i}{v}\leq1,
\]
and the coefficients of the remaining terms satisfy
\[
\begin{aligned}
 (p-1)\frac{\lambda_i}{v}+b+1-\Lambda p
 &=-2p\alpha-(p-1)\left(1-\frac{\lambda_i}{v}\right),\\
 b\left(\frac{\lambda_i}{v}-\Lambda\right)
 &=p\alpha^2+p\alpha\left(1-\frac{\lambda_i}{v}\right).
\end{aligned}
\]
Substitution, followed by $v_i/v=(\log u)_i+\alpha F_i$, yields
\[
\begin{aligned}
 \mathcal Q
 &=\sum_i p\left|\frac{v_i}{v}-\alpha F_i\right|^2
 +\sum_i\left(1-\frac{\lambda_i}{v}\right)
 \left[-(p-1)\operatorname{Re}\left(\frac{v_i}{v}F_{\bar i}\right)
          +p\alpha|F_i|^2\right]\\
 &=\sum_i p|(\log u)_i|^2
 +\sum_i\left(1-\frac{\lambda_i}{v}\right)
 \left[-(p-1)\operatorname{Re}((\log u)_iF_{\bar i})
          +\alpha|F_i|^2\right].
\end{aligned}
\]
For each $i$, the last bracket equals
\[
 \alpha\left|F_i-\frac{p-1}{2\alpha}(\log u)_i\right|^2
 -\frac{(p-1)^2}{4\alpha}|(\log u)_i|^2.
\]
Hence one gets the complete square identity
\begin{equation}\label{eq:ht8}
\begin{split}
 \mathcal Q
 ={}&\sum_i\alpha\left(1-\frac{\lambda_i}{v}\right)
       \left|F_i-\frac{p-1}{2\alpha}(\log u)_i\right|^2\\
 &+\sum_i\left[p-\frac{(p-1)^2}{4\alpha}
                  \left(1-\frac{\lambda_i}{v}\right)\right]
                  |(\log u)_i|^2.
\end{split}
\end{equation}
Choose
\[
 \alpha=K+p,\qquad \Lambda-K=p+1,\qquad
 p-\frac{(p-1)^2}{4\alpha}\left(1-\frac{\lambda_i}{v}\right)
 \geq\frac p2.
\]
Thus \eqref{eq:trace-combination} gives
\begin{equation}\label{eq:ht9}
\begin{split}
 &\frac p2\int_Xe^{-p\alpha F}v^p
                   |\nabla_\varphi\log u|_\varphi^2\,d\mu
 +(p+1)\int_Xe^{-p\alpha F}v^{p+1}\,d\mu\\
 &\qquad\leq(\alpha+1)(-\underline R+\delta)
                       \int_Xe^{-p\alpha F}v^p\,d\mu.
\end{split}
\end{equation}

H\"older's inequality gives
\[
 \int_Xe^{-p\alpha F}v^p\,d\mu
 \leq\left(\int_Xe^{-p\alpha F}v^{p+1}\,d\mu\right)^{p/(p+1)}
      \left(\int_Xe^{-p\alpha F}\,d\mu\right)^{1/(p+1)}.
\]
Dropping the first nonnegative term in \eqref{eq:ht9} yields
\[
 \int_Xe^{-p\alpha F}v^{p+1}\,d\mu
 \leq\left(\frac{(\alpha+1)(-\underline R+\delta)}{p+1}\right)^{p+1}
       \int_Xe^{-p\alpha F}\,d\mu.
\]
Since $\mu(X)=1$, multiplying by $e^{p\alpha\sup F}$ gives
\begin{equation}\label{eq:ht10}
 \int_X(\tr_\varphi\chi+\delta)^{p+1}\,d\mu
 \leq\left(\frac{(\alpha+1)(-\underline R+\delta)}{p+1}\right)^{p+1}
       e^{p\alpha\operatorname{osc}F},\qquad \alpha=K+p.
\end{equation}
Let $\delta\downarrow0$. This proves \eqref{eq:ht2} for $p\geq2$;
for $1\leq p<2$, use $\|\tr_\varphi\chi\|_p\leq
\|\tr_\varphi\chi\|_2$.
\end{proof}

For the canonical family, take $Z=\mathbb \C P^N$ and the scaled
Fubini--Study metric whose pullback is $\chi$.
Theorem~\ref{thm:bounds} give
$|\underline R_\varepsilon|+\operatorname{osc}F_\varepsilon\leq C$.
The target is smooth even when $Y$ is singular, so \eqref{eq:ht2}
holds uniformly for every fixed finite $p$.

Now we prove the gradient bound of $F$. 
\begin{lemma}\label{lem:gradient}
Under the same assumptions in Lemma~\ref{lem:trace-lp},
\begin{equation}\label{eq:c-18b}
 \sup_X|\nabla_\varphi F|_\varphi^2\leq C'.
\end{equation}
\end{lemma}

\begin{proof}
Put $w=|\nabla_\varphi F|_\varphi^2$. The Bochner formula gives 
\begin{equation}\label{eq:gradient-bochner}
 \Delta_\varphi w
 =\sum_{i,j}(|F_{ij}|^2+|F_{i\bar j}|^2)-2\operatorname{Re}\sum_i(\tr_\varphi\chi)_iF_{\bar i}-\sum_i\lambda_i|F_i|^2
 -\sum_{i,j}F_{i\bar j}F_jF_{\bar i}.
\end{equation}
Direct computation gives
\[w_i=\sum_j(F_{ji}F_{\bar j}+F_jF_{i\bar j}),\]
 \[e^{-bF}\Delta_\varphi e^{bF}=b(-\underline R-\tr_\varphi\chi)+b^2w,\]
 \[2e^{-bF}\operatorname{Re}\sum_i(e^{bF})_iw_{\bar i}=2b\operatorname{Re}\sum_{i,j}F_{ij}F_{\bar i}F_{\bar j}
   +2b\sum_{i,j}F_{i\bar j}F_jF_{\bar i},\]
\[ \sum_{i,j}|F_{ij}|^2
+2b\operatorname{Re}\sum_{i,j}F_{ij}F_{\bar i}F_{\bar j}+b^2w^2=\sum_{i,j}|F_{ij}+bF_iF_j|^2.
\]
Thus
\[
\begin{split}
 e^{-bF}\Delta_\varphi(e^{bF}w)
 =&\sum_{i,j}|F_{ij}+bF_iF_j|^2+\sum_{i,j}|F_{,i\bar j}|^2+(2b-1)\sum_{i,j}F_{i\bar j}F_jF_{\bar i}-2\operatorname{Re}\sum_i(\tr_\varphi\chi)_iF_{\bar i}\\
 &-\sum_i\lambda_i|F_i|^2+b(-\underline R-\tr_\varphi\chi)w.
\end{split}
\]
The mixed Hessian has no sign; its coefficient vanishes for
$b=\frac12$. This is \cite[equations~(4.3)--(4.4)]{CC17}, with
reference Ricci form $-\chi$. Set
\[
 u=1+e^{F/2}|\nabla_\varphi F|_\varphi^2.
\]
Then
\begin{equation}\label{eq:gradient-weighted-exact}
\begin{split}
 \Delta_\varphi u
 =&e^{F/2}\sum_{i,j}\left(|F_{ij}+\tfrac12F_iF_j|^2+|F_{i\bar j}|^2\right)-2e^{F/2}\operatorname{Re}\sum_i(\tr_\varphi\chi)_iF_{\bar i}-e^{F/2}\sum_i\lambda_i|F_i|^2\\
 &+\tfrac12(-\underline R-\tr_\varphi\chi)(u-1).
\end{split}
\end{equation}
Since $0\leq\lambda_i\leq\tr_\varphi\chi$, dropping the Hessian
squares gives
\begin{equation}\label{eq:c-18c}
 \Delta_\varphi u \geq-2e^{F/2}\operatorname{Re}\langle\nabla_\varphi\tr_\varphi\chi,\nabla_\varphi F\rangle_\varphi
 -\left(\tfrac12\underline R+\tfrac32\tr_\varphi\chi\right)(u-1).
\end{equation}

For $p\geq2$, testing \eqref{eq:c-18c} with $u^{p-1}$ gives
\begin{equation}\label{eq:gradient-tested}
\begin{split}
 &(p-1)\int_Xu^{p-2}|\nabla_\varphi u|_\varphi^2\,d\mu\\
 &\quad\leq2\int_Xe^{F/2}u^{p-1}\operatorname{Re}\langle\nabla_\varphi\tr_\varphi\chi,\nabla_\varphi F\rangle_\varphi\,d\mu\\
&\qquad+\int_X\left(\tfrac12\underline R+\tfrac32\tr_\varphi\chi\right)
                      (u-1)u^{p-1}\,d\mu.
\end{split}
\end{equation}
For the first integral on the right,
\begin{equation}\label{eq:gradient-q-parts}
\begin{split}
 &2\int_Xe^{F/2}u^{p-1}\operatorname{Re}\langle\nabla_\varphi\tr_\varphi\chi,\nabla_\varphi F\rangle_\varphi\,d\mu\\
 &\quad=-2\int_X(\tr_\varphi\chi)e^{F/2}u^{p-1}
                      (-\underline R-\tr_\varphi\chi)\,d\mu\\
 &\qquad-\int_X(\tr_\varphi\chi)(u-1)u^{p-1}\,d\mu\\
 &\qquad-2(p-1)\int_X(\tr_\varphi\chi)e^{F/2}u^{p-2}\operatorname{Re}\langle\nabla_\varphi u,\nabla_\varphi F\rangle_\varphi\,d\mu.
\end{split}
\end{equation}
Because $e^{F/2}|\nabla_\varphi F|_\varphi
=e^{F/4}\sqrt{u-1}\leq C\sqrt u$, Young's inequality gives
\[
\begin{split}
 &2(p-1)\left|\int_X(\tr_\varphi\chi)e^{F/2}u^{p-2}
 \operatorname{Re}\langle\nabla_\varphi u,\nabla_\varphi F\rangle_\varphi\,d\mu\right|\\
 &\quad\leq\frac{p-1}{2}\int_Xu^{p-2}|\nabla_\varphi u|_\varphi^2\,d\mu
       +C(p-1)\int_X(\tr_\varphi\chi)^2u^{p-1}\,d\mu.
\end{split}
\]
The terms without $\nabla_\varphi u$ combine to
\[
\begin{split}
 &2\underline R\int_X(\tr_\varphi\chi)e^{F/2}u^{p-1}\,d\mu
 +2\int_X(\tr_\varphi\chi)^2e^{F/2}u^{p-1}\,d\mu\\
 &\quad+\frac12\int_X(\tr_\varphi\chi)(u-1)u^{p-1}\,d\mu
 +\frac{\underline R}{2}\int_X(u-1)u^{p-1}\,d\mu\\
 &\qquad\leq C\int_X\bigl(1+(\tr_\varphi\chi)^2\bigr)u^p\,d\mu,
\end{split}
\]
using $u\geq1$.
Absorbing the gradient term and using
$|\nabla_\varphi u^{p/2}|_\varphi^2
=\frac{p^2}{4}u^{p-2}|\nabla_\varphi u|_\varphi^2$ yields
\begin{equation}\label{eq:c-18d}
 \int_X|\nabla_\varphi u^{p/2}|_\varphi^2\,d\mu
 \leq Cp^2\int_X\bigl(1+(\tr_\varphi\chi)^2\bigr)u^p\,d\mu.
\end{equation}

Let $\theta=q/[r(q-1)]\in(0,1)$, where $q>1$. H\"older and interpolation inequality give
\begin{equation}\label{eq:iteration-interpolation}
\begin{split}
 \int_X\bigl(1+(\tr_\varphi\chi)^2\bigr)u^p\,d\mu
 &\leq C\|u^{p/2}\|_{L^{2r/(r-1)}}^2
 \leq C\|u^{p/2}\|_{L^{2q}}^{2\theta}\|u^{p/2}\|_{L^2}^{2(1-\theta)}\\
 &\leq\tau\|u^{p/2}\|_{L^{2q}}^2
       +C\tau^{-\theta/(1-\theta)}\|u^{p/2}\|_{L^2}^2,
 \qquad 0<\tau\leq1.
\end{split}
\end{equation}
Apply the uniform Sobolev inequality~\ref{thm:sobolev} to $u^{p/2}$. 
Choosing $\tau=(2Cp^2)^{-1}$ in \eqref{eq:iteration-interpolation}
and absorbing half the left side gives
\begin{equation}\label{eq:iteration-step}
 \|u\|_{L^{pq}}\leq\bigl(Cp^{2/(1-\theta)}\bigr)^{1/p}\|u\|_{L^p},
 \qquad p\geq2.
\end{equation}
For $p_j=2q^j$, both $\sum_jp_j^{-1}$ and
$\sum_jp_j^{-1}\log p_j$ converge. Iteration and interpolation yield
\[
 \|u\|_{L^\infty}\leq C\|u\|_{L^2}
 \leq C\|u\|_{L^\infty}^{1/2}\|u\|_{L^1}^{1/2},
 \qquad \|u\|_{L^\infty}\leq C^2\|u\|_{L^1}.
\]
Finally, by \eqref{eq:ht1},
\[
 \int_Xwd\mu
 =-\int_XF(-\underline R-\tr_\varphi\chi)d\mu
 \leq2|\underline R|\|F\|_{L^\infty},
 \qquad \|u\|_{L^1}\leq1+e^{\|F\|_{L^\infty}/2}\int_Xwd\mu\leq C.
\]
This proves \eqref{eq:c-18b}.
\end{proof}

Now we can prove the desired $C^2$-estimates in the base direction. 
\begin{lemma}\label{lem:trace}
Under the assumption of Lemma~\ref{lem:trace-lp}, 
\[
 \tr_\varphi\chi\leq C',\qquad \omega_\varphi\geq(C')^{-1}\chi.
\]
\end{lemma}

\begin{proof}
Set $v=1+\tr_\varphi\chi$. Taking $\delta=1$, $b=0$ in
\eqref{eq:ht6} gives, for $p\geq2$,
\[
 \int_Xv^p\sum_i\left[
 p\left|\frac{v_i}{v}\right|^2+
 \left((p-1)\frac{\lambda_i}{v}+1\right)
       \operatorname{Re}\left(\frac{v_i}{v}F_{\bar i}\right)\right]d\mu
 \leq K\int_Xv^{p+1}\,d\mu.
\]
Since $0\leq\lambda_i/v\leq1$,
\[
 \left|\sum_i\left((p-1)\frac{\lambda_i}{v}+1\right)
       \operatorname{Re}\left(\frac{v_i}{v}F_{\bar i}\right)\right|
 \leq p\frac{|\nabla_\varphi v|_\varphi}{v}|\nabla_\varphi F|_\varphi
 \leq\frac p2\frac{|\nabla_\varphi v|_\varphi^2}{v^2}+Cp.
\]
Absorption gives
\[
 \frac p2\int_Xv^{p-2}|\nabla_\varphi v|_\varphi^2\,d\mu
 \leq K\int_Xv^{p+1}\,d\mu+Cp\int_Xv^p\,d\mu,
\]
and hence
\begin{equation}\label{eq:q1}
 \int_X|\nabla_\varphi v^{p/2}|_\varphi^2\,d\mu
 \leq Cp^2\int_X(1+\tr_\varphi\chi)v^p\,d\mu.
\end{equation}
For fixed $r>q/(q-1)$, the coefficient $1+\tr_\varphi\chi$ has a
uniform $L^r$ bound. The interpolation and iteration in
\eqref{eq:iteration-interpolation}--\eqref{eq:iteration-step} therefore give
\[
 \|1+\tr_\varphi\chi\|_{L^\infty}
 \leq C\|1+\tr_\varphi\chi\|_{L^2}\leq C'.
\]
Finally, $\lambda_i\leq\sum_j\lambda_j=\tr_\varphi\chi\leq C'$,
so $\chi\leq C'\omega_\varphi$. 
\end{proof}

\section{Uniform convergence of the normalized volume density}
In this section, we first show that one can always take limit(up to subsequence) of $\varphi_\varepsilon$ and $F_\varepsilon$ in $L^p$ for any finite $p$. Then we analyze the convergence of the Mabuchi energy along the cscK path. Using the minimizing property of the cscK metrics to identify the limit of $F_\varepsilon$ as the canonical volume density $f^*\varphi_{\rm can}$. Then based on the uniform Sobolev inequality of Guo--Phong--Song--Sturm~\cite{GPSS23}, we apply the local Moser iteration to improve the $L^p$ convergence of $F_\varepsilon$ to $L^\infty$ convergence away from the singular fibers.

With the estimates in previous section, we observe that one can take some weak limit up to subsequence. 

\begin{lemma}\label{lem:energy}
The functions $F_\varepsilon$, $\varphi_\varepsilon$, have uniformly bounded normalized Dirichlet energy:
\[
\int_X |\nabla F_\varepsilon|_{\omega_\varepsilon}^2\,d\mu_\varepsilon+\int_X |\nabla \varphi_\varepsilon|_{\omega_\varepsilon}^2\,d\mu_\varepsilon\le C.
\]
\end{lemma}
\begin{proof}
$\chi\ge 0$ gives
\[
-\underline R_\varepsilon=\int_X\operatorname{tr}_{\omega_\varepsilon}\chi\,d\mu_\varepsilon\ge0.
\]
Note that  $\underline R_\varepsilon\to-m$. Integration by parts yields
\[
\int_X|\nabla F_\varepsilon|_{\omega_\varepsilon}^2\,d\mu_\varepsilon
=-\int_XF_\varepsilon(-\underline R_\varepsilon-\operatorname{tr}_{\omega_\varepsilon}\chi)\,d\mu_\varepsilon
\le2|\underline R_\varepsilon|\|F_\varepsilon\|_\infty.
\]
Likewise,
\[
\Delta_{\omega_\varepsilon}\varphi_\varepsilon=n-\operatorname{tr}_{\omega_\varepsilon}\chi_\varepsilon,
\qquad
\int_X\operatorname{tr}_{\omega_\varepsilon}\chi_\varepsilon\,d\mu_\varepsilon=n,
\]
where the last identity follows from $[\chi_\varepsilon]=[\omega_\varepsilon]$. Both $\chi_\varepsilon$ and $\omega_\varepsilon$ are positive. Thus
\[
\int_X|\nabla\varphi_\varepsilon|_{\omega_\varepsilon}^2\,d\mu_\varepsilon\le2n\|\varphi_\varepsilon\|_\infty.
\]
\end{proof}

\begin{lemma}\label{lem:compactness}
Suppose
\[
 \|u_\varepsilon\|_{L^\infty}\leq C,\qquad
 \int_X|\nabla u_\varepsilon|_{\omega_\varepsilon}^2\,d\mu_\varepsilon\leq C.
\]
Then:
\begin{enumerate}
\item\label{item:lem-compactness-1} $\|u_\varepsilon\|_{W^{1,1}(X,\mathchanged{\omega})}\leq C$. Hence every sequence has a subsequence converging in every finite
$L^p(X,\Omega)$.
\item\label{item:lem-compactness-2} For any compact subset $K\Subset Y^\circ$, put
\[
 \bar u_\varepsilon(y)=V_F^{-1}\int_{X_y}u_\varepsilon\mathchanged{\omega_y}^d.
\]
Then
\begin{align*}
 \int_{f^{-1}(K)}|\nabla_Vu_\varepsilon|_{\mathchanged{\omega_y}}\Omega
 &\leq C_K\sqrt\varepsilon,\\
 \int_{f^{-1}(K)}|u_\varepsilon-f^*\bar u_\varepsilon|\Omega
 &\leq C_K\sqrt\varepsilon.
\end{align*}
Here $\nabla_V u$ is the covariant derivative in the fiber direction. 
Hence every limit is constant on almost every smooth connected fiber.
For $d=0$, the fibers on $Y^\circ$ are points and this last assertion
is automatic.
\end{enumerate}
\end{lemma}

\begin{proof}
1. By Cauchy--Schwarz inequality
\[
 |\nabla u_\varepsilon|_{\mathchanged{\omega}}^2
 \leq|\nabla u_\varepsilon|_{\omega_\varepsilon}^2
                 \tr_{\mathchanged{\omega}}\omega_\varepsilon.
\]
Since $\|F_\varepsilon\|_\infty\le C$, $\mu_\varepsilon, \Omega, \mathchanged{\omega}^n$ are all equivalent,
\begin{align*}
\int_X\tr_{\mathchanged{\omega}}\omega_\varepsilon\,\Omega
 &\leq Cn\int_X\omega_\varepsilon\wedge\mathchanged{\omega}^{n-1}
 =Cn(2\pi c_1(K_X)+\varepsilon [\mathchanged{\omega}])\cdot [\mathchanged{\omega}]^{n-1}\leq C,\\
 \int_X|\nabla u_\varepsilon|_{\mathchanged{\omega}}\Omega
 &\leq C\left(\int_X|\nabla u_\varepsilon|_{\omega_\varepsilon}^2d\mu_\varepsilon\right)^{1/2}
       \left(\int_X\tr_{\mathchanged{\omega}}\omega_\varepsilon\,\Omega\right)^{1/2}\leq C.
\end{align*}
Rellich compactness gives $h_{\varepsilon_j}\to h$ in $L^1$; then
\[
 \|h_{\varepsilon_j}-h\|_{L^p}^p
 \leq(2C)^{p-1}\|h_{\varepsilon_j}-h\|_{L^1}\longrightarrow0.
\]

2. Take $K\Subset K'\Subset Y^\circ$.
For $d>0$,
\[
 [\omega_\varepsilon|_{X_y}]=\varepsilon[\mathchanged{\omega_y}],\qquad
 \int_{X_y}\tr_{\mathchanged{\omega_y}}(\omega_\varepsilon|_{X_y})\mathchanged{\omega_y}^d
 =d\varepsilon V_F.
\]
On $f^{-1}(K')$, $\Omega, \mathchanged{\omega}^d\wedge f^*\chi_Y^m$ are equivalent;
hence fiber integration yields
\[
 \int_{f^{-1}(K)}\tr_{\mathchanged{\omega_y}}(\omega_\varepsilon|_{X_y})\Omega
 \leq C_K\varepsilon.
\]
The dual norm decreases on restriction, so
\begin{align*}
 |\nabla_Vu_\varepsilon|_{\mathchanged{\omega_y}}^2
 &\leq|\nabla h_\varepsilon|_{\omega_\varepsilon}^2
              \tr_{\mathchanged{\omega_y}}(\omega_\varepsilon|_{X_y}),\\
 \int_{f^{-1}(K)}|\nabla_Vu_\varepsilon|_{\mathchanged{\omega_y}}\Omega
 &\leq C_K\left(\int_X|\nabla u_\varepsilon|_{\omega_\varepsilon}^2d\mu_\varepsilon\right)^{1/2}
              \varepsilon^{1/2}\leq C_K\sqrt\varepsilon.
\end{align*}
The smooth compact family of connected fibers has a uniform
$L^1$ Poincar\'e inequality:
\[
 \int_{X_y}|u_\varepsilon-\bar u_\varepsilon(y)|\mathchanged{\omega_y}^d
 \leq C_K\int_{X_y}|\nabla_V u_\varepsilon|_{\mathchanged{\omega_y}}\mathchanged{\omega_y}^d.
\]
Integrating proves the second estimate. Passing to distributions,
$\nabla_V u=0$; connectedness gives fiberwise constancy.
\end{proof}

Applying the energy and compactness estimates to the two functions in the cscK system gives the following subsequential limits.

\begin{proposition}\label{prop:volume}

Every sequence $\varepsilon_j\downarrow0$ has a subsequence, still denoted $\varepsilon_j$, with the following properties.
\begin{enumerate}
\item\label{item:volume-potentials}
\[
F_{\varepsilon_j}\longrightarrow F_\infty,\qquad
\varphi_{\varepsilon_j}\longrightarrow\varphi_\infty
\quad\text{ in every }L^p(X,\Omega)\quad (p<\infty).
\]
On $X^\circ$, the two limits are almost everywhere pullbacks of bounded measurable functions on $Y^\circ$. \item\label{item:volume-density}
\[
\frac{\omega_{\varepsilon_j}^n}{V_{\varepsilon_j}}\longrightarrow e^{F_\infty}\Omega
\quad\text{in total variation},
\qquad
\int_XF_{\varepsilon_j}e^{F_{\varepsilon_j}}\Omega\longrightarrow
\int_XF_\infty e^{F_\infty}\Omega.
\]
\item\label{item:volume-current} The weak current limit $\chi+dd^c\varphi_\infty$ is positive and has bounded local potentials. 
\end{enumerate}
\end{proposition}

\begin{proof}
Apply Lemmas~\ref{lem:energy} and~\ref{lem:compactness} and a common subsequence. Exponentiation and the function $x\mapsto xe^x$ are Lipschitz on the common bounded range. Hence their $L^1$ convergence follows from that of $F_{\varepsilon_j}$, proving total variation and entropy convergence. Positivity follows by passing to the current limit in $\chi+\varepsilon_j\mathchanged{\omega}+dd^c\varphi_{\varepsilon_j}\ge0$. The uniform potential bound passes to an almost everywhere representative of the limit, whose quasi-plurisubharmonic regularization has the corresponding bounded local potentials.
\end{proof}

We recall some facts about bounded psh functions on singular space. 
\begin{enumerate}
    \item For a bounded $\chi$-psh function $u$, one has \[
u|_{X_y}\in\PSH(X_y)\Longrightarrow u|_{X_y}\equiv h(y),
 \qquad h\in\PSH(Y^\circ,\chi_Y)\cap L^\infty.
    \]By \cite[Chatper I, Theorem 5.24]{D12}, $h$ can be extended to a bounded psh function $Y_{\rm reg}$. Thus $h$ is a weak psh function on a normal space $Y$. By \cite[Theorem 1.2]{GGZ24}, the extension $g(y)=\limsup_{Y^\circ\ni y'\to y}h(y')$ is psh on $Y$. The extension does not increase the $L^\infty$-norm of the psh function. 
\item Let $h\in \PSH(Y,\chi_Y)\cap L^\infty$. There exists a sequence of smooth functions $h_j\in \PSH(\C P^N,\ell^{-1}\omega_{FS})$ which decrease pointwisely on $\C P^N$ so that $\lim h_j=h$ on $Y$. See \cite[Corollary C]{CGZ13}.
\item Let $\pi:\widetilde{Y}\to Y$ be a resolution which is obtained by composition of blowups with smooth centers. Then the pullback identidies the bounded psh spaces:\[
\PSH(Y,\chi_Y)\cap L^\infty\longleftrightarrow\PSH(\widetilde{Y},\pi^*\chi_Y)\cap L^\infty.
\]See \cite{D85}. Moreover, by monotone convergence theorem of Bedford-Taylor, for bounded $u_j\downarrow u\in \PSH(Y,\chi_Y)\cap L^\infty$, \[
\MA_Y(u_j)\rightharpoonup\MA_Y(u), \quad E_Y(u_j)\downarrow E_Y(u),
\] where \[
\MA_Y(u)=V_Y^{-1}(\chi_Y+\ddc u)^m, \qquad 
E_Y(u)=\frac1{(m+1)V_Y}\sum_{j=0}^m\int_Yu(\chi_Y+dd^c u)^j\chi_Y^{m-j}.
\]
The convergence also holds for mixed products. And with this identification of psh functions, we can regard all functionals on $Y$ as those on its resolutions. 

\item For $\varphi\in \PSH(\widetilde{Y},\pi^*\chi_Y)\cap L^\infty$, $g\in C^0(\widetilde{Y})$, and the envelope $P_{\pi^*\chi_Y}(h):=\sup\{u\in \PSH:u\leq h\}$, \[\left.
\frac{d}{dt}\right|_{t=0}E_{\pi^*\chi_Y}(P_{\pi^*\chi_Y}(\varphi+tg))=\int_{\widetilde{Y}}g\MA_{\pi^*\chi_Y}(\varphi).
\]This is proved in \cite[Lemma~4.2]{BBGZ13}.
\end{enumerate}

We want to use the variational method to show that $F_\varepsilon\to f^*\varphi_{\rm can}.$ Let $\nu=e^{F}f_*\Omega$, where $F$ is a bounded function. If $u$ is a bounded psh function that maximizes \[
E_Y(\varphi)-\int_Y\varphi d\nu,
\]
then $MA_Y(u)=\nu$. To see this, lift $\nu$ to $\widetilde{Y}$ with zero measure on the exceptional set. Set $\tilde{u}=\pi^*u\in \PSH(\widetilde{Y},\pi^*\chi_Y)\cap L^\infty$, and $w_s=P_{\pi^*\chi_Y}(\tilde{u}+sg)$ for a bounded continuous $g$. Hence $w_s\in \PSH(\widetilde{Y},\pi^*\chi_Y)\cap L^\infty$. Since $\tilde{u}$ maximizes the functional, we get \[
E_{\pi^*\chi_Y}(w_s)-\int_{\widetilde{Y}}\tilde{u}d\nu-s\int_{\widetilde{Y}}gd\nu\leq E_{\chi_Y}(u)-\int_Yud\nu.
\]
Apply the variational formula, we get $\MA_Y(u)=\nu$.

In the following,  we suppose a subsequence $\varphi_\varepsilon\to f^*u$, and $F_\varepsilon\to f^*F$ in $L^p(X,\Omega)$, and set $\nu=e^Ff_*\Omega$, \[
    E_\varepsilon(\varphi)=\frac{1}{(n+1)V_{\varepsilon}}\sum_{j=0}^n\int_X\varphi(\chi_\varepsilon+\ddc\varphi)^j\chi_{\varepsilon}^{n-j}.
    \]
    
\begin{lemma}\label{lem:am-limit}
\leavevmode
\begin{enumerate}
        \item $E_\varepsilon(\varphi_\varepsilon)\to E_Y(u)$.
        \item $u$ maximizes the functional $E_Y(\varphi)-\int_Y\varphi d\nu$.
        \item $\MA_Y(u)=\nu$.
    \end{enumerate}
\end{lemma}
\begin{proof}
    Fix a smooth $\varphi\in \PSH(Y,\chi_Y)$, one has \[
    E_{\varepsilon}(f^*\varphi)=\frac{1}{V_{\varepsilon}}\int_X\int_0^1f^*\varphi(\chi+\varepsilon\mathchanged{\omega}+t\ddc f^*\varphi)^ndt
    \]
    Since powers of a pull-back form above degree $m$ vanish, uniformly for $t\in [0,1]$, \[
\begin{aligned}
        &\int_X f^*\varphi\,
       \bigl(\chi+tdd^c f^*\varphi+\varepsilon\mathchanged{\omega}\bigr)^n\\
 &\quad=\sum_{k=d}^n\binom nk\varepsilon^k
       \int_X f^*\bigl(\varphi(\chi_Y+tdd^c\varphi)^{n-k}\bigr)
                      \wedge\mathchanged{\omega}^k\\
 &\quad=\binom nd\varepsilon^dV_F
       \int_Y\varphi(\chi_Y+tdd^c\varphi)^m+O_\varphi(\varepsilon^{d+1}),\\
 &V_\varepsilon=\binom nd\varepsilon^dV_FV_Y+O(\varepsilon^{d+1}).
    \end{aligned}
    \]
Here we used that $f_*(\mathchanged{\omega}^d)=V_F$, so division cancels the factor
$\binom nd\varepsilon^dV_F$:\[
 \begin{aligned}
 E_\varepsilon(f^*\varphi)
 &=\frac1{V_Y}\int_0^1\int_Y
       \varphi\bigl((1-t)\chi_Y+t\chi_\varphi\bigr)^m\,dt+O_\varphi(\varepsilon)\\
 &=\frac1{(m+1)V_Y}\sum_{j=0}^m
       \int_Y\varphi\chi_\varphi^j\wedge\chi_Y^{m-j}
          +O_\varphi(\varepsilon)
 =E_Y(\varphi)+O_\varphi(\varepsilon).
\end{aligned}
\]
A similar argument also shows that \[
\frac{(\chi+\varepsilon\mathchanged{\omega}+\ddc f^*\varphi)^n}{V_\varepsilon}\to \frac{\mathchanged{\omega}^d\wedge f^*(\chi_Y+\ddc\varphi)^m}{V_F V_Y},
\]
and integration along fiber gives \[
f_*\left(\frac{\mathchanged{\omega}^d\wedge f^*(\chi_Y+\ddc\varphi)^m}{V_F V_Y}\right)=V_Y^{-1}(\chi_Y+\ddc\varphi)^m.
\]
Note that the energy $E_\varepsilon$ is concave, so \[
\begin{aligned}
&E_{\varepsilon}(\varphi_\varepsilon)\leq E_\varepsilon(f^*\varphi)+\int_X(\varphi_\varepsilon-f^*\varphi)\frac{(\chi+\varepsilon\mathchanged{\omega}+\ddc f^*\varphi)^n}{V_\varepsilon}\\
&E_{\varepsilon}(\varphi_\varepsilon)\geq E_\varepsilon(f^*\varphi)+\int_X(\varphi_\varepsilon-f^*\varphi)d\mu_\varepsilon.
\end{aligned}
\]
Taking $\varepsilon\to 0$, we have \[
\begin{aligned}
    &\limsup_{\varepsilon\to 0}E_\varepsilon(\varphi_\varepsilon)\leq E_Y(\varphi)+\int_Y(u-\varphi)\MA_Y(\varphi),\\
    &\liminf_{\varepsilon\to 0}E_\varepsilon(\varphi_\varepsilon)\geq E_Y(\varphi)+\int_Y(u-\varphi)d\nu.
\end{aligned}
\]
Here, we used \[
 \begin{aligned}
\left|\int_X(\varphi_\varepsilon-f^*u)\,d\mu_\varepsilon\right|
 &\leq C\|\varphi_\varepsilon-f^*u\|_{L^1(\Omega)}\longrightarrow0,\\
 \left|\int_X f^*(u-\varphi)
        \bigl(e^{F_\varepsilon}-e^{f^*F}\bigr)\Omega\right|
 &\leq\|u-\varphi\|_\infty
       \|e^{F_\varepsilon}-e^{f^*F}\|_{L^1(\Omega)}\longrightarrow0,
 \end{aligned}
\]Since $\varphi_\varepsilon$ and $F_\varepsilon$ are uniformly bounded. 

Now we choose a sequence of bounded smooth $\varphi_j\in \PSH(\mathbb{CP}^N,\ell^{-1}\omega_{FS})$ that $\varphi_j\downarrow u$. First let $\varepsilon\to 0$ with $j$ fixed. Since $\varphi_j-u\geq 0$, we have \[
 \begin{aligned}
 E_Y(\varphi_j)+\int_Y(u-\varphi_j)\,d\nu
 &\leq\liminf_{\varepsilon\to0}E_\varepsilon(\varphi_\varepsilon)\\
 &\leq\limsup_{\varepsilon\to0}E_\varepsilon(\varphi_\varepsilon)\\
 &\leq E_Y(\varphi_j)+\int_Y(u-\varphi_j)\MA_Y(\varphi_j)
 \leq E_Y(\varphi_j).
 \end{aligned}
\]
Let $j\to \infty$, \[
E_Y(\varphi_j)\downarrow E_Y(u),\qquad \int_Y(u-\varphi_j)d\nu\to 0.
\]
Hence $E_{\varepsilon}(\varphi_\varepsilon)\to E_Y(u)$. Moreover, we have \[
E_Y(u)\geq E_Y(\varphi)+\int_Y(u-\varphi)d\nu
\]
for any smooth  $\varphi\in \PSH(Y,\chi_Y)$. Apply the decreasing approximation again, it follows that the above inequality holds for any bounded $\varphi\in \PSH(Y,\chi_Y)$. The equation follows from the variational formula discussed before. 

\end{proof}

\begin{lemma}\label{lem:mix-current}\leavevmode
\begin{enumerate}
\item\label{item:mixed-local} On $Y^\circ$,
\begin{equation}\label{eq:c-6}
f_*\left(\frac{\omega_\varepsilon^{n-1}}{V_\varepsilon}\right)\rightharpoonup\frac{m}{nV_Y}(\chi_Y+dd^cu)^{m-1}.
\end{equation}
\item\label{item:mixed-global} On $Y$,
\begin{equation}\label{eq:c-7}
f_*\left(\frac{\chi\wedge\omega_\varepsilon^{n-1}}{V_\varepsilon}\right)\rightharpoonup\frac{m}{nV_Y}\chi_Y\wedge(\chi_Y+dd^cu)^{m-1}.
\end{equation}
\end{enumerate}
\end{lemma}
\begin{proof}
If $n=1$, then $m=1$, $d=0$, and $f$ is an isomorphism of smooth
curves. Both conclusions follow from $V_\varepsilon\to V_Y$.
For the rest of the proof assume $n\geq2$.
Fix smooth $\varphi$ with $\chi_Y+dd^c \varphi\ge\delta\chi_Y$, put $\widehat\omega=\chi+\varepsilon\mathchanged{\omega}+\ddc f^*\varphi$, $h=\varphi_\varepsilon-f^*\varphi$, and 
\[
I_\varepsilon=V_\varepsilon^{-1}\int_Xh(\widehat\omega^n-\omega_\varepsilon^n)
=V_\varepsilon^{-1}\sum_{k=0}^{n-1}\int_Xi\partial h\wedge\bar\partial h\wedge\omega_\varepsilon^k\wedge\widehat\omega^{n-1-k}.
\]
We claim that 
\begin{equation}\label{eq:c-8}
I_\varepsilon\to I_Y(u,\varphi):=\int_Y(u-\varphi)(\operatorname{MA}_Y(\varphi)-\operatorname{MA}_Y(u)).
\end{equation}
Recall the binomial expansion in Lemma~\ref{lem:am-limit}:\[
\frac{\widehat{\omega}^n}{V_\varepsilon}=\frac{\mathchanged{\omega}^d\wedge f^*(\chi_Y+\ddc \varphi)^m}{V_FV_Y}+O_{\varphi}(\varepsilon)\Omega\leq C_{\varphi}\Omega.
\]
Write $h=(\varphi_\varepsilon-f^*u)+(f^*u-\varphi)$. Then 
\[
 \begin{aligned}
 &\left|\frac1{V_\varepsilon}\int_Xh\,\widehat\omega^n
                   -\int_Y(u-\varphi)\MA_Y(\varphi)\right|\\
 &\quad\leq
 \left|\frac1{V_\varepsilon}\int_X
             (\varphi_\varepsilon-f^*u)\widehat\omega^n\right|
 +\left|\int_X f^*(u-\varphi)
       \left(\frac{\widehat\omega^n}{V_\varepsilon}
             -\frac{\mathchanged{\omega}^d\wedge f^*\chi_\varphi^m}{V_FV_Y}\right)\right|\\
 &\quad\leq C_\varphi\|\varphi_\varepsilon-f^*u\|_{L^1(\Omega)}
       +C_\varphi\varepsilon\|u-\varphi\|_\infty\longrightarrow0.
 \end{aligned}
\]
For the second integral, Lemma~\ref{lem:am-limit} has already proved
$\MA_Y(u)=e^Ff_*\Omega$. Thus
\[
 \begin{aligned}
&\left|\frac1{V_\varepsilon}\int_Xh\,\omega_\varepsilon^n
-\int_Y(u-\varphi)\MA_Y(u)\right|\\
 &\quad\leq
 \left|\int_X(\varphi_\varepsilon-f^*u)e^{F_\varepsilon}\Omega\right|+\left|\int_Xf^*(u-\varphi)(e^{F_\varepsilon}-e^{f^*F})\Omega\right|\\
 &\quad\leq C\|\varphi_\varepsilon-f^*u\|_{L^1(\Omega)}+\|u-\varphi\|_{L^\infty}\|e^{F_\varepsilon}-e^{f^*F}\|_{L^1(\Omega)}
       \longrightarrow0.
 \end{aligned}
\]
The last limit uses $e^{F_\varepsilon}\leq C$ and
$\|e^{F_\varepsilon}-e^{f^*F}\|_{L^1}
\leq C\|F_\varepsilon-f^*F\|_{L^1}\to0$.
Subtracting these two limits proves the claim. 

Next, for smooth compactly supported $(1,1)$ test form $\Xi$ on $Y^\circ$, pair it with $f_*(V_\varepsilon^{-1}(\widehat{\omega}^{n-1}-\omega_\varepsilon^{n-1}))$. The pairing is a sum of the following terms
\[
V_\varepsilon^{-1}\int_Xi\partial h\wedge\bar\partial f^*\Xi\wedge T_k,
\qquad T_k=\omega_\varepsilon^k\wedge\widehat\omega^{n-2-k}.
\]
All covectors of $\bar\partial f^*\Xi$ are horizontal with bounded coefficients in $\chi$.  Cauchy--Schwarz inequality gives
\begin{equation}\label{eq:c-9}
\left|V_\varepsilon^{-1}\int i\partial h\wedge\bar\partial f^*\Xi\wedge T_k\right|
\le C_\Xi
\left(V_\varepsilon^{-1}\int i\partial h\wedge\bar\partial h\wedge\chi\wedge T_k\right)^{1/2}
\left(V_\varepsilon^{-1}\int\chi^2\wedge T_k\right)^{1/2}.
\end{equation}
Since $\widehat\omega\ge\delta\chi$, the first squared factor is at most $\delta^{-1}I_\varepsilon$. The other squared factor is
\[
 \begin{aligned}
 V_\varepsilon^{-1}\int_X\chi^2\wedge T_k
 &=\frac{[\chi]^2[\chi_\varepsilon]^{n-2}}{V_\varepsilon}\\
 &=\frac{\binom{n-2}{d}\varepsilon^dV_FV_Y+O(\varepsilon^{d+1})}
         {\binom nd\varepsilon^dV_FV_Y+O(\varepsilon^{d+1})}
 \longrightarrow\frac{m(m-1)}{n(n-1)}\qquad(m\geq2).
 \end{aligned}
\]
 If $m=1$, this factor is identically
zero because $\chi^2=0$. Therefore
\begin{equation}\label{eq:c-10}
\left|\left\langle\Xi,f_*\frac{\omega_\varepsilon^{n-1}-\widehat\omega_{\varepsilon,\varphi}^{n-1}}{V_\varepsilon}\right\rangle\right|
\le C_\Xi\delta^{-1/2}\sqrt{I_\varepsilon}.
\end{equation}

For $0<\delta<1$, put $v_\delta=(1-\delta)u$, so that
$\chi_{v_\delta}=(1-\delta)\chi_u+\delta\chi_Y\geq\delta\chi_Y$.
For bounded psh functions, we have 
\begin{equation}\label{eq:c-11}
 \begin{split}
 I_Y(u,v_\delta)
 &=\frac{\delta^2}{V_Y}\sum_{k=0}^{m-1}
    \int_Yi\partial u\wedge\bar\partial u
              \wedge\chi_u^k\wedge\chi_{v_\delta}^{m-1-k}\\
 &\leq2m\|u\|_{L^\infty}\delta^2.
 \end{split}
\end{equation}
Indeed, for $T=\chi_u^k\wedge\chi_{v_\delta}^{m-1-k}$, bounded
Bedford--Taylor integration by parts on the resolution yields
\[
 \begin{aligned}
 0\leq\int_Yi\partial u\wedge\bar\partial u\wedge T
 &=-\int_Yu\,dd^cu\wedge T
   =\int_Yu(\chi_Y-\chi_u)\wedge T\\
 &\leq\|u\|_\infty\int_Y(\chi_Y+\chi_u)\wedge T
   =2\|u\|_{L^\infty} V_Y.
 \end{aligned}
\]
This proves the uniform bound in \eqref{eq:c-11}.

Choose smooth $u_j\downarrow u$, and set
$\varphi_{\delta,j}=(1-\delta)u_j$.
The three limits are taken in the order
$\varepsilon\to0$, then $j\to\infty$, then $\delta\to0$.
For fixed $\delta,j$, direct computation gives
\[
 f_*\left(\frac{(\chi+\varepsilon\mathchanged{\omega}+\ddc f^*\varphi_{j,\delta})^{n-1}}
                         {V_\varepsilon}\right)
 \longrightarrow
 \frac{\binom{n-1}{d}V_F}{\binom nd V_FV_Y}
                 ( \chi_Y+\ddc \varphi_{\delta,j})^{m-1}
 =\frac{m}{nV_Y}(\chi_Y+\ddc \varphi_{\delta,j})^{m-1}.
\]
Hence \eqref{eq:c-8} and \eqref{eq:c-10} imply
\[
 \limsup_{\varepsilon\to0}
 \left|\left\langle\Xi,\,
 f_*\frac{\omega_\varepsilon^{n-1}}{V_\varepsilon}
       -\frac{m}{nV_Y}\chi_{\varphi_{\delta,j}}^{m-1}
 \right\rangle\right|
 \leq C_\Xi\delta^{-1/2}\sqrt{I_Y(u,\varphi_{\delta,j})}.
\]
For fixed $\delta$, monotone convergence of the Bedford-Taylor products implies 
\[
I_Y(u,\varphi_{\delta,j})\longrightarrow I_Y(u,v_\delta),\qquad
 (\chi_Y+\ddc{\varphi_{\delta,j}})^{m-1}\rightharpoonup(\chi_Y+\ddc{v_\delta})^{m-1}.
\]
Letting $j\to\infty$ in the preceding estimate and using
\eqref{eq:c-11}, we obtain
\[
 \limsup_{\varepsilon\to0}
 \left|\left\langle\Xi,\,
 f_*\frac{\omega_\varepsilon^{n-1}}{V_\varepsilon}
       -\frac{m}{nV_Y}(\chi_Y+\ddc{v_\delta})^{m-1}
 \right\rangle\right|
 \leq C_\Xi\sqrt\delta.
\]
Finally,
\[
 (\chi_Y+\ddc{v_\delta})^{m-1}
 =\sum_{k=0}^{m-1}\binom{m-1}{k}
      (1-\delta)^k\delta^{m-1-k}
                     (\chi_Y+\ddc u)^k\wedge\chi_Y^{m-1-k}
 \rightharpoonup(\chi_Y+\ddc u)^{m-1}.
\]
Let $\delta\to0$ to obtain \eqref{eq:c-6}. These formulas also
hold when $m=1$.

\emph{Proof of \textup{(2)}.}
The bounds for $F_\varepsilon$ and
$\tr_{\omega_\varepsilon}\chi$ imply
\[
0\leq\frac{\chi\wedge\omega_\varepsilon^{n-1}}{V_\varepsilon}
 =\frac1n(\tr_{\omega_\varepsilon}\chi)e^{F_\varepsilon}\Omega
 \leq C\Omega.
\]
Thus, for every bounded nonnegative measurable $b$ on $Y$,
\[
 \begin{aligned}
 0\leq\frac1{V_\varepsilon}\int_Xf^*b\,
                       \chi\wedge\omega_\varepsilon^{n-1}
 &\leq C\int_X f^*b\,\Omega=C\int_Y b\,f_*\Omega.
 \end{aligned}
\]
Equivalently, as positive measures on $Y$,
\[
 0\leq
 f_*\!\left(\frac{\chi\wedge\omega_\varepsilon^{n-1}}{V_\varepsilon}\right)
 \leq C f_*\Omega .
\]

For $b\in C^\infty_c(Y^\circ)$, apply \textup{(i)} with the
test form $\Xi=b\chi_Y$. Since $\chi=f^*\chi_Y$, this gives
\[
 \begin{aligned}
 \frac1{V_\varepsilon}\int_X f^*b\,
                          \chi\wedge\omega_\varepsilon^{n-1}
 &=\left\langle b\chi_Y,\,
                 f_*\frac{\omega_\varepsilon^{n-1}}{V_\varepsilon}
   \right\rangle\\
 &\longrightarrow
 \frac{m}{nV_Y}\int_{Y^\circ}b\,\chi_Y\wedge(\chi_Y+\ddc u)^{m-1}.
 \end{aligned}
\]
For nonnegative $b$, the preceding bound passes to this limit.
Consequently,
\[
 0\leq\frac{m}{nV_Y}\chi_Y\wedge(\chi_Y+\ddc u)^{m-1}
 \leq C f_*\Omega\quad\text{on }Y^\circ.
\]
Recall $D=Y\setminus Y^\circ$. Its inverse image is proper analytic,
and bounded Bedford--Taylor products give no mass to proper
analytic sets on a resolution.
Therefore
\[
 (f_*\Omega)(D)=\int_{f^{-1}(D)}\Omega=0,\qquad
 \int_D\chi_Y\wedge(\chi_Y+\ddc u)^{m-1}=0.
\]
The same measure inequality thus holds on all of $Y$.

It remains to allow test functions without compact support in
$Y^\circ$. Fix $b\in L^1(Y,f_*\Omega)$. The restriction of
$f_*\Omega$ to $Y^\circ$ is a smooth positive volume form; it has
finite total mass and gives no mass to $D$. Smooth approximation
and exhaustion of $Y^\circ$ therefore give
\[
 b_j\in C^\infty_c(Y^\circ),\qquad
 \|b_j-b\|_{L^1(f_*\Omega)}\longrightarrow0.
\]
The two measure bounds imply
\[
 \begin{aligned}
 &\left|\frac1{V_\varepsilon}\int_X f^*(b-b_j)\,
                          \chi\wedge\omega_\varepsilon^{n-1}\right|
 \leq C\|b-b_j\|_{L^1(f_*\Omega)},\\
 &\left|\frac{m}{nV_Y}\int_Y(b-b_j)\,
                          \chi_Y\wedge(\chi_Y+\ddc u)^{m-1}\right|
 \leq C\|b-b_j\|_{L^1(f_*\Omega)}.
 \end{aligned}
\]
First let $\varepsilon\to0$ with $j$ fixed, using the convergence
already proved for $b_j$. The triangle inequality gives
\[
 \begin{aligned}
 &\limsup_{\varepsilon\to0}
 \left|\frac1{V_\varepsilon}\int_Xf^*b\,
                          \chi\wedge\omega_\varepsilon^{n-1}
       -\frac{m}{nV_Y}\int_Yb\,\chi_Y\wedge(\chi_Y+\ddc u)^{m-1}\right|\\
 &\qquad\leq2C\|b-b_j\|_{L^1(f_*\Omega)}
       \xrightarrow{j\to\infty}0 .
 \end{aligned}
\]
This proves \eqref{eq:c-7}.

\end{proof}

A similar concavity argument as in Lemma~\ref{lem:am-limit}, together with the above convergence of mixed products, shows convergence of the $J$-functional as follows. The $J$-functional is defined as \[
\begin{aligned}
J_{\chi,\varepsilon}(\varphi)&:=\frac{1}{n V_\varepsilon}\sum_{j=0}^{n-1}\int_X\varphi\chi(\chi_\varepsilon+\ddc \varphi)^{j}\chi_\varepsilon^{n-1-j},\\
J_{\chi_Y,Y}(\varphi)&:=\frac{1}{m V_Y}\sum_{j=0}^{m-1}\int_Y\varphi\chi_Y(\chi_Y+\ddc \varphi)^{j}\chi_Y^{m-1-j}.
\end{aligned}
\]

\begin{lemma}
    \label{lem:mixed-energy}
$nJ_{\chi,\varepsilon}(\varphi_\varepsilon)\to mJ_{\chi_Y,Y}(u)$.
\end{lemma}
\begin{proof}
  Direct computation shows \[
    \frac{d}{dt}|_{t=0}J_{\chi,\varepsilon}(\varphi+th)=V_{\varepsilon}^{-1}\int_Xh\chi\wedge(\chi+\varepsilon\mathchanged{\omega}+\ddc\varphi)^{n-1},
    \]
   and \[
   \frac{d}{dt}|_{t=0}J_{\chi_Y,Y}(\varphi+th)=V_Y\int_Yh\chi_Y\wedge(\chi_Y+\ddc\varphi)^{m-1}.
   \]
Moreover \[
\frac{d^2}{dt^2}|_{t=0}J_{\chi,\varepsilon}\leq 0, \qquad \frac{d^2}{dt^2}|_{t=0}J_{\chi_Y,Y}\leq 0.\]
Then similar expansion in the proof of Lemma~\ref{lem:am-limit} gives, for fixed smooth psh function $\varphi\in \PSH(Y,\chi_Y)$, \[
J_{\chi,\varepsilon}(f^*\varphi)\to \frac{m}{n}J_{\chi_Y,Y}(\varphi).
\]
Note that $|F_\varepsilon|, \tr_{\omega_\varepsilon}\chi$ are uniformly bounded,\[
\left|\frac{1}{V_\varepsilon}\int_X(\varphi_\varepsilon-f^*u)\chi\wedge\omega_\varepsilon^{n-1}\right|\leq C\|\varphi_{\varepsilon}-f^*u\|_{L^2(\Omega)}\|\tr_{\omega_\varepsilon}\chi\|_{L^2(\mu_\varepsilon)}\to 0.
\]
By Lemma~\ref{lem:mix-current}, \[
\frac{1}{V_\varepsilon}\int_Xf^*(u-\varphi)\chi\wedge\omega_\varepsilon^{n-1}\to \frac{m}{nV_Y}\int_Y(u-\varphi)\chi_Y\wedge(\chi_Y+\ddc u)^{m-1}. 
\]
Hence, a similar concavity argument as in Lemma~\ref{lem:am-limit} gives, \[
\limsup_{\varepsilon\to 0}nJ_{\chi,\varepsilon}(\varphi_\varepsilon)\leq mJ_{\chi_Y,Y}(\varphi)+\frac{m}{V_Y}\int_Y(u-\varphi)\chi_Y\wedge(\chi_Y+\ddc\varphi)^{m-1},
\]
\[
\liminf_{\varepsilon\to 0}nJ_{\chi,\varepsilon}(\varphi_\varepsilon)\geq mJ_{\chi_Y,Y}(\varphi)+\frac{m}{V_Y}\int_Y(u-\varphi)\chi_Y\wedge(\chi_Y+\ddc u)^{m-1}.
\]

We choose $\varphi_j\downarrow u$, with $\varphi_j\in \PSH(\C P^N,\ell^{-1}\omega_{FS})$ being smooth. Monotone convergence of the mixed energies gives the desired convergence. 
\end{proof}
Now, we use the Mabuchi energy to characterize the weak limit of the volume density. 
\begin{lemma}\label{lem:canonical-volume-density}
    Any subsequential limit $f^*F$ satisfies $F=\varphi_{\rm can}$ $f_*\Omega$-almost everywhere. Thus $F_\varepsilon\to f^*\varphi_{\rm can}$ in $L^p(X,\Omega)$ for any finite $p.$
\end{lemma}
\begin{proof}
    For a smooth $\chi_\varepsilon$-psh function $\varphi$, the normalized Mabuchi energy is defined as \begin{equation}
        \label{eq:mabuchi-energy}
        \cM_\varepsilon(\varphi):=\frac{1}{V_\varepsilon}\int_X\log{\frac{(\chi_\varepsilon+\ddc\varphi)^n}{V_\varepsilon\Omega}}(\chi_\varepsilon+\ddc\varphi)^n+nJ_{\chi,\varepsilon}(\varphi)+\underline{R}_\varepsilon E_{\varepsilon}(\varphi)
    \end{equation}
The Mabuchi energy is convex along geodesics, and cscK metrics minimize it.
Now we apply this to two family of metrics:\begin{enumerate}
    \item For $\hw_\varepsilon$, the potential $\psi_\varepsilon\to f^*\varphi_{\rm can}$, and $\nu_{\rm can}=e^{f^*\varphi_{\rm can}}f_*\Omega$, and $\tr_{\hw_\varepsilon}\chi\leq C$.
    \item For $\omega_\varepsilon$, by passing to subsequence, the potential $\varphi_\varepsilon\to f^*u$, and $\nu_u=e^Ff_*\Omega=\MA_Y(u)$, and $\tr_{\omega_\varepsilon}\chi\leq C$.
\end{enumerate}
On $Y$, one has \[
mJ_{\chi_Y,Y}(u)=(m+1)E_Y(u)-\int_Yu\MA_Y(u).
\]
Since $\underline{R}_\varepsilon\to -m$ and the bounded $\log$-density converges in $L^1$, we get \[
\cM_\varepsilon(\varphi_\varepsilon)\to \int_YFe^{f^*F}f_*\Omega+mJ_{\chi_Y,Y}(u)-mE_Y(u)=\int_YFe^{f^*F}f_*\Omega+E_Y(u)-\int_Yu\MA_Y(u).
\]
The same argument in Lemma~\ref{lem:am-limit},\ref{lem:mix-current},\ref{lem:mixed-energy} also works for $\hw_\varepsilon$, and we get \[
\cM_\varepsilon(\psi_\varepsilon)\to \int_Y\varphi_{\rm can}e^{\varphi_{\rm can}}f_*\Omega+E_Y(\varphi_{\rm can})-\int_Y\varphi_{\rm can}\MA_{Y}(\varphi_{\rm can}).
\]
Since $\omega_\varepsilon$ is the unique cscK metric, we have $\cM_\varepsilon(\varphi_\varepsilon)\leq \cM_{\varepsilon}(\psi_\varepsilon)$.

Note that from concavity, we also have \[
E_Y(\varphi_{\rm can})\leq E_Y(u)+\int_Y(\varphi_{\rm can}-u)\MA_Y(u).
\]

Therefore \[
\begin{aligned}
    0\geq& \int_Y(Fe^{f^*F}-\varphi_{\rm can}e^{\varphi_{\rm can}})f_*\Omega+E_Y(u)-E_Y(\varphi_{\rm can})-\int_Yu\MA_Y(u)+\int_Y\varphi_{\rm can}\MA_Y(\varphi_{\rm can})\\
    =&\int_X(F-\varphi_{\rm can}) e^Ff_*\Omega+E_Y(u)-E_Y(\varphi_{\rm can})-\int_Y(u-\varphi_{\rm can})\MA_Y(u)\\
    \ge &\int_Y(F-\varphi_{\rm can})e^Ff_*\Omega.
\end{aligned}
\]
Here we used that $\MA_Y(u)=e^Ff_*\Omega$ by Lemma~\ref{lem:am-limit}, and the equation of the canonical metric on the base $\MA_Y(\varphi_{\rm can})=e^{\varphi_{\rm can}}f_*\Omega$. 

Now the difference of the entropy term \[
\int_Y(F-\varphi_{\rm can})e^Ff_*\Omega=\int_Y(F-\varphi_{\rm can})e^{F-\varphi_{\rm can}}\MA_Y(\varphi_{\rm can})=\int_Y(s\log s-s+1)\MA_Y(\varphi_{\rm can})\geq 0,
\]where we let $s=e^{F-\varphi_{\rm can}}$ and used the facts that $s\log s-s+1\geq 0$ and $=$ attains only at $s=1$. Therefore, every inequality is equality. In particular, $s=1$ almost everywhere, i.e. \[
F=\varphi_{\rm can}.
\]
\end{proof}

Next, we show that the base direction of $\omega_\varepsilon$ converges to $f^*\omega_{\rm can}$ in the following weak $L^p$-sense. We started with a general results concerning entropy convergence in the same cohomology class. 
\begin{lemma}\label{lem:entropy-to-metric}
Let $\omega,\widehat\omega$ be cohomologous K\"ahler forms on a compact
complex $n$-fold. Set
$V=\int_X\omega^n$, $\mu=\omega^n/V$, and $\nu=\widehat\omega^n/V$.
\begin{enumerate}
\item\label{item:entropy-bound}
\[
 \int_X|\widehat\omega-\omega|_\omega\,d\mu
 \leq2\sqrt{n\Ent_\nu(\mu)},\qquad
 \text{ where }\Ent_\nu(\mu)=\int_X\log\frac{d\mu}{d\nu}d\mu.
\]
\item\label{item:entropy-convergence}
Suppose $\mu_j=\rho_j\Omega$, $\nu_j=\widehat\rho_j\Omega$, where
$C^{-1}\leq\rho_j,\widehat\rho_j\leq C$ and
$\|\rho_j-\widehat\rho_j\|_{L^1(\Omega)}\to0$. Then
\[
 \int_X|\widehat\omega_j-\omega_j|_{\omega_j}\,d\mu_j\to0,
 \qquad
 \int_X|\omega_j-\widehat\omega_j|_{\widehat\omega_j}\,d\nu_j\to0.
\]
For every $\eta>0$, outside a set of $\Omega$-measure tending to zero,
\[
 (1-\eta)\omega_j\leq\widehat\omega_j\leq(1+\eta)\omega_j.
\]
\end{enumerate}
\end{lemma}

\begin{proof}
Note the following elementary inequality 
\[
 (\sqrt t-1)^2\leq t-1-\log t,\qquad
 (\sqrt t+1)^2\leq2(t+1),\qquad
 |t-1|=|\sqrt t-1|(\sqrt t+1).
\]
It follows that 
\[
 |\widehat\omega-\omega|_\omega^2
 \leq2(\tr_\omega\widehat\omega+n)
 \left(\tr_\omega\widehat\omega-n-
                    \log\frac{\widehat\omega^n}{\omega^n}\right).
\]
The first factor is bounded due to the cohomology class
\[
\int_X\tr_\omega\widehat\omega d\mu
 =\frac{n}{V}\int_X\widehat\omega\wedge\omega^{n-1}=n,
\]
and 
\[
 \int_X\left(\tr_\omega\widehat\omega-n-
                 \log\frac{\widehat\omega^n}{\omega^n}\right)d\mu
 =\Ent_\nu(\mu).
\]
By Cauchy--Schwarz inequality, one gets (i).

For (2),
\[
 0\leq\Ent_{\nu_j}(\mu_j)
 \leq C^2\|\rho_j-\widehat\rho_j\|_{L^1(\Omega)}\to0,
\]
and likewise with the measures interchanged. Finally,
\[
 |\widehat\omega_j-\omega_j|_{\omega_j}\leq\eta
 \quad\Longrightarrow\quad
 -\eta\omega_j\leq\widehat\omega_j-\omega_j\leq\eta\omega_j;
\]
Markov's inequality gives the last assertion.
\end{proof}

\begin{corollary}\label{cor:Lp-convergence-in-base-direction}
 For every smooth $(1,0)$-form $\alpha$ on $Y^\circ$, and $K\Subset Y^\circ$
\[|f^*\alpha|_{\omega_\varepsilon}^2
 \longrightarrow f^*|\alpha|_{\omega_{\rm can}}^2
 \quad\text{in }L^p(f^{-1}(K),\Omega),\quad p<\infty.
\]
\end{corollary}

\begin{proof}
Apply Lemma~\ref{lem:entropy-to-metric} to $\omega_\varepsilon$ and
$\mathchanged{\tilde\omega}_\varepsilon$. Their normalized densities are uniformly
bounded above and below and have the same total-variation limit.
The uniform convergence of the metric in regular part in Theorem~\ref{thm:NKRF-convergence} gives
\[
 (1-o_K(1))(f^*\omega_{\rm can}+\varepsilon\omega_{\rm SRF})\leq\mathchanged{\tilde\omega}_\varepsilon
 \leq(1+o_K(1))(f^*\omega_{\rm can}+\varepsilon\omega_{\rm SRF}).
\]
Use the $\omega_{\rm SRF}$-orthogonal complement of the vertical
tangent bundle. The reference form has zero mixed term; its
horizontal restriction is $f^*\omega_{\rm can}+O_K(\varepsilon)$
and its vertical restriction is $\varepsilon\omega_{{\rm SRF},y}$.
Since $f^*\alpha$ vanishes on vertical vectors, duality gives
\[|f^*\alpha|_{\mathchanged{\tilde\omega}_\varepsilon}^2\longrightarrow f^*|\alpha|_{\omega_{\rm can}}^2\quad\hbox{uniformly on }f^{-1}(K).
\]
When $d=0$, this follows directly from convergence of positive
metrics on $f^{-1}(K)$.
For $0<\eta<1$, outside a set of measure tending to zero,
\[
 (1-\eta)\omega_\varepsilon\leq\mathchanged{\tilde\omega}_\varepsilon\leq(1+\eta)\omega_\varepsilon,
\]
so duality gives
\[
 (1-\eta)|f^*\alpha|_{\mathchanged{\tilde\omega}_\varepsilon}^2\leq|f^*\alpha|_{\omega_\varepsilon}^2\leq(1+\eta)|f^*\alpha|_{\mathchanged{\tilde\omega}_\varepsilon}^2.
\]
The comparison-family limit gives convergence in measure as
$\eta\downarrow0$. Since
$|f^*\alpha|_{\omega_\varepsilon}^2
\leq C_\alpha\tr_{\omega_\varepsilon}\chi\leq C_\alpha$,
the convergence holds in every finite $L^p$. 
\end{proof}

Next we use the Moser iteration to improve the $L^p$
 convergence of $F_\varepsilon$ to $L^\infty$-convergence away from singular fibers.  We summarize it as a lemme here. 

Fix a holormorhic coordinate $z$ on a neighborhood of  $B_{2r}\Subset Y^\circ$, where $B_{2r}=\{|z|\leq 2r\}$. Set $\omega_E=\ddc |z|^2$ and assume $\omega_E\leq C\chi_Y$. Fix a smooth function $0\le \eta\le 1$ on $\R$, with $\eta=1$ on $(-\infty,0]$ and $\eta=0$ on $[\frac12,\infty)$. For $r\leq s<t\leq 2r$, set $
\eta_{s,t}(z)=\eta\left(\frac{|z|^2-s^2}{t^2-s^2}\right)$, and $\xi_{s,t}=f^*\eta_{s,t}.$ Then \[
\eta_{s,t}\in C^\infty_c(B_t), \quad \eta_{s,t}=1 \text{ on }B_s, \quad |\nabla\eta|^2_{\omega_E}\leq \frac{C}{(t-s)^2},
\]
and on the pullback, \[
|\nabla\xi_{s,t}|_{\omega_\varepsilon}^2\leq \frac{C}{(t-s)^2}\tr_{\omega_\varepsilon}f^*\omega_E\leq 
\frac{C}{(t-s)^2}.
\]
The following estimates may depend on the fixed coordinate chart but uniform on $\varepsilon$. Recall that we use $q$ to denote the exponent of uniform Sobolev inequality obtained by Guo--Phong--Song--Strum \cite{GPSS23}.  

\begin{lemma}\label{lem:iteration}
Fix $q/(q-1)<p_*<\infty$. 
Every smooth real function $U$ on $f^{-1}(B_{2r})$ satisfies
\begin{equation}
\|U\|_{L^\infty(f^{-1}(B_r))}
 \leq C\left(\|U\|_{L^2(f^{-1}(B_{2r}),\mu_\varepsilon)}+\|\Delta_{\omega_\varepsilon}U\|_{L^{p_*}(f^{-1}(B_{2r}),\mu_\varepsilon)}
 \right).
\end{equation}
The constant depends only on $p_*$ and the fixed chart. 
\end{lemma}

\begin{proof}
Suppress $\varepsilon$ and use
$\mu_\varepsilon$ for all integrals and norms.
We first estimate $U_+=\max\{U,0\}$. For $\delta>0$, set
\[
 v=U_++\|\Delta_\omega U\|_{L^{p_*}(f^{-1}(B_{2r}))}+\delta,
 \qquad
 \left\|\frac{|\Delta_\omega U|}{v}\right\|_{L^{p_*}(f^{-1}(B_{2r}))}
 \leq1.
\]
$-\Delta_\omega U_+\leq|\Delta_\omega U|$ in the weak sense, 
hence $-\Delta_\omega v\leq|\Delta_\omega U|$.
For $p\geq2$ and $\xi=\xi_{s,t}$, use the nonnegative test function
$\xi^2v^{p-1}$.
All integrals in the next display are over
$f^{-1}(B_t)$, and all differential norms use $\omega_\varepsilon$.
Integration by parts and Young's inequality give
\[
 \begin{aligned}
 (p-1)\int\xi^2v^{p-2}|\nabla v|^2
 &\leq2\int\xi v^{p-1}|\nabla\xi|\,|\nabla v|
       +\int\frac{|\Delta_\omega U|}{v}\,\xi^2v^p\\
 &\leq\frac{p-1}{2}\int\xi^2v^{p-2}|\nabla v|^2
       +\frac{2}{p-1}\int v^p|\nabla\xi|^2
       +\int\frac{|\Delta_\omega U|}{v}\,\xi^2v^p.
 \end{aligned}
\]
The cutoff bound and the uniform Sobolev inequality imply
\[
 \|\xi v^{p/2}\|_{L^{2q}}^2
 \leq Cp^2\left[
       \bigl(1+(t-s)^{-2}\bigr)\int_{f^{-1}(B_t)}v^p+\int_{f^{-1}(B_t)}\frac{|\Delta_\omega U|}{v}(\xi v^{p/2})^2\right].
\]
Put $p_*'=p_*/(p_*-1)<q$ and
$\theta=q/[p_*(q-1)]\in(0,1)$. For $h=\xi v^{p/2}$,
H\"older's inequality and interpolation give
\[
 \begin{aligned}
\int\frac{|\Delta_\omega U|}{v}h^2&\leq\|h\|_{L^{2p_*'}}^2
\leq\|h\|_{L^{2q}}^{2\theta}\|h\|_{L^2}^{2(1-\theta)}\\
&\leq\tau\|h\|_{L^{2q}}^2+C_\theta\tau^{-\theta/(1-\theta)}\|h\|_{L^2}^2\qquad(\tau>0).
 \end{aligned}
\]
Choose $\tau=(2Cp^2)^{-1}$ and absorb the first term. Since
$\|h\|_{L^2}^2\leq\int_{f^{-1}(B_t)}v^p$ and $\xi=1$ on $f^{-1}(B_s)$,
\begin{equation}\label{eq:dg4}
 \|v\|_{L^{pq}(f^{-1}(B_s))}
 \leq\left[Cp^{2/(1-\theta)}\bigl(1+(t-s)^{-2}\bigr)\right]^{1/p}\|v\|_{L^p(f^{-1}(B_t))}.
\end{equation}

Set $r_l=r(1+2^{-l})$ and $p_l=2q^l$. Taking
$(s,t,p)=(r_{l+1},r_l,p_l)$ in \eqref{eq:dg4} yields
\begin{equation}\label{eq:dg5}
 \|v\|_{L^{p_{l+1}}(f^{-1}(B_{r_{l+1}}))}
 \leq\left(C_r4^lp_l^{2/(1-\theta)}\right)^{1/p_l}\|v\|_{L^{p_l}(f^{-1}(B_{r_l}))}.
\end{equation}
Consequently,
\[\prod_{l=0}^\infty\left(C_r4^lp_l^{2/(1-\theta)}\right)^{1/p_l}<\infty,
\qquad\|v\|_{L^\infty(f^{-1}(B_r))}
\leq C\|v\|_{L^2(f^{-1}(B_{2r}))}.
\]
Since $U_+\leq v$ and $\mu_\varepsilon(X)=1$,
\[
 \|U_+\|_{L^\infty(f^{-1}(B_r))}
 \leq C\left(\|U_+\|_{L^2(f^{-1}(B_{2r}))}
       +\|\Delta_\omega U\|_{L^{p_*}(f^{-1}(B_{2r}))}+\delta\right).
\]
Let $\delta\downarrow0$. The same argument for $-U$ estimates
$U_-=(-U)_+$ and proves the desired estimate.
\end{proof}

 \begin{lemma}\label{lem:Linfty-density}
     For any compact $K\Subset X^\circ$, $\|F_\varepsilon-f^*\varphi_{\rm can}\|_{L^\infty(K)}\to 0$.
 \end{lemma}
\begin{proof}
    Fix $q/(q-1)<p_*<\infty$. On any coordinate ball $B_{2r}\Subset Y^\circ$, set $U=F_\varepsilon-f^*\varphi$. Then one has \[
    \Delta_{\omega_\varepsilon}U=-\underline{R}_\varepsilon-\tr_{\omega_\varepsilon}\chi-\tr_{\omega_{\varepsilon}}(f^*(\omega_{\rm can}-\chi_Y))=-\underline{R}_\varepsilon-\tr_{\omega_\varepsilon}f^*\omega_{\rm can}.
    \]
    Choose a smooth $\omega_{\rm can}$-unitary $(1,0)$-froms $\alpha_1,...,\alpha_m$ on a neighborhood of $B_{2r}$. Then by Lemma~\ref{cor:Lp-convergence-in-base-direction}, and $|F_\varepsilon|\leq C$, \[
    \omega_{\rm can}=\sqrt{-1}\sum_i\alpha_j\wedge\bar{\alpha}_j, \quad \tr_{\omega_\varepsilon}f^*\omega_{\rm can}\to m \text{ in } L^{p^*}(f^{-1}(B_{2r}),\mu_\varepsilon)
    \]
Combined with $\underline{R}_\varepsilon\to -m$, it follows that \[
\|\Delta_{\omega_\varepsilon}U\|_{L^{p^*}(f^{-1}(B_{2r}),\mu_\varepsilon)}\leq |\underline{R}_{\varepsilon}+m|+\|\tr_{\omega_\varepsilon}f^*\omega_{\rm can}-m\|_{L^{p^*}(f^{-1}(B_{2r}),\mu_\varepsilon)}\to 0.
\]
Moreover, we showed in Lemma~\ref{lem:canonical-volume-density} that $F_\varepsilon\to f^*\varphi\
_{\rm can}$ in $L^p$ for any finite $p$. Hence $\|U\|_{L^{p^*}(f^{-1}(B_{2r}),\mu_\varepsilon)}\to 0$. For any given $K\Subset X^\circ$, choose finitely many coordinate balls such that \[
f(K)\subset \cup_{j=1}^MB_{r_{j}}, \quad B_{2r_j}\Subset Y^\circ.
\]
Applying Lemma~\ref{lem:iteration}, we get the desired $L^\infty$ convergence. 
\end{proof}

\section{Sharp Schwarz-type estimate}

\begin{lemma}\label{lem:sharp-schwarz-estimate}
    For every $K\Subset X^\circ$, $\omega_\varepsilon\geq (1-o(1))f^*\omega_{\rm can}$
\end{lemma}
\begin{proof}
   Work over $B_{2r}\Subset Y^\circ$, with base coordinates
$y^1,\ldots,y^m$. For
$a=\sum_{\alpha=1}^m a_\alpha\,dy^\alpha$ with constant coefficients
and $|a|_{\omega_{\rm can}}\leq1$ on $B_{2r}$, set
\[\omega_a=\sqrt{-1}a\wedge\bar a,\qquad u_\varepsilon=\tr_{\omega_\varepsilon}f^*\omega_a,\qquad
u_0=\tr_{\omega_{\rm can}}\omega_a.
\]
We have
\[
 0\leq u_\varepsilon\leq C,\qquad
 \|u_\varepsilon-f^*u_0\|_{L^2(f^{-1}(B_{2r}),\mu_\varepsilon)}
 \longrightarrow0.
\]
Suppress $\varepsilon$ and put
\[
 h=\sum_{\alpha=1}^m a_\alpha(y^\alpha\circ f),\qquad
 f^*a=\partial h,\qquad
f^*\omega_a=\sqrt{-1}\,\partial h\wedge\bar\partial\bar h.
\]
Fix $x\in f^{-1}(B_{2r})$ and choose holomorphic coordinates
$z^1,\ldots,z^n$ normal for $\omega$ at $x$:
\[
 \omega=\sqrt{-1}g_{i\bar j}\,dz^i\wedge d\bar z^j,\qquad
 g_{i\bar j}(x)=\delta_{ij},\qquad
 \partial_k g_{i\bar j}(x)=0.
\]
Write $h_i=\partial_i h$, $h_{ik}=\partial_k\partial_i h$, with
$\sum_jg^{i\bar j}g_{k\bar j}=\delta^i_k$. Since $h$ is holomorphic,
\[
 u=\sum_{i,j}g^{i\bar j}h_i\overline{h_j},\qquad
 \partial_{\bar l}h_i=0,\qquad
 \partial_k\overline{h_j}=0.
\]
Straightforward computation gives  
\[
 \partial_k u
 =\sum_{i,j}(\partial_k g^{i\bar j})h_i\overline{h_j}
  +\sum_{i,j}g^{i\bar j}h_{ik}\overline{h_j}.
\]
\[
 u_{k\bar l}=\sum_{i,j}(\partial_k\partial_{\bar l}g^{i\bar j})h_i\overline{h_j}+\sum_ih_{ik}\overline{h_{il}}.
\]
At 
$x$, we have
\[
 \begin{aligned}
\partial_k\partial_{\bar l}g^{i\bar j}
 &=-\partial_k\partial_{\bar l}g_{j\bar i}
   =R_{j\bar i k\bar l}.
 \end{aligned}
\]
Thus, using $\Ric(\omega)=-\chi-\ddc F$,
\[
 \begin{aligned}
 \Delta_\omega u
 &=\sum_{i,k}|h_{ik}|^2+\sum_{i,j}\Ric(\omega)_{j\bar i}h_i\overline{h_j}\\
 &=\sum_{i,k}|h_{ik}|^2
   -\sum_{i,j}\chi_{j\bar i}h_i\overline{h_j}-\sum_{i,j}F_{j\bar i}h_i\overline{h_j}.
 \end{aligned}
\]
Thus
\begin{equation}\label{eq:sharp-schwarz-bochner}
 \begin{split}
 \Delta_\omega u
 &=|\nabla^{1,0}_\omega(f^*a)|_\omega^2
   -\langle f^*\omega_a,\chi\rangle_\omega\\
 &\qquad-\langle f^*\omega_a,\ddc F\rangle_\omega.
 \end{split}
\end{equation}
It is also \eqref{eq:map-bochner-coordinates} for the holomorphic
map $h$ to flat $\C$. A unitary change of the normal coordinates
diagonalizes $\chi$, with eigenvalues $\lambda_i\geq0$, so
\[
 0\leq\langle f^*\omega_a,\chi\rangle_\omega=\sum_i\lambda_i|h_i|^2\leq\left(\sum_i\lambda_i\right)\left(\sum_i|h_i|^2\right)
 =(\tr_\omega\chi)u\leq Cu.
\]
In particular,
\[
 \Delta_\omega u\geq-Cu-\langle f^*\omega_a,\ddc F\rangle_\omega.
\]
We retain the square in \eqref{eq:sharp-schwarz-bochner} below.
The smoothness of $u_0$ on the fixed chart gives
\[
 |\nabla f^*u_0|_\omega^2\leq C\tr_\omega\chi\leq C,
 \qquad
 |\Delta_\omega f^*u_0|
 =|\tr_\omega f^*(\ddc u_0)|\leq C\tr_\omega\chi\leq C.
\]
Consequently,
\begin{equation}\label{eq:sharp-schwarz-excess}
 \Delta_\omega(u-f^*u_0)\geq|\nabla^{1,0}_\omega(f^*a)|_\omega^2-C-\langle f^*\omega_a,\ddc F\rangle_\omega.
\end{equation}

Set $v=(u-f^*u_0)_+$, $\xi=\xi_{s,t}$ and
$\psi=\xi^2v^{p-1}$, where $p\geq3$.
All following integrals use $d\mu_\varepsilon$ on $f^{-1}(B_t)$,
and all differential norms use $\omega$. Testing
\eqref{eq:sharp-schwarz-excess} with $\psi$,
\[
 \begin{aligned}
 &(p-1)\int\xi^2v^{p-2}|\nabla v|^2+\int\psi|\nabla^{1,0}_\omega(f^*a)|^2\\
 &\quad\leq2\int\xi v^{p-1}|\nabla\xi|\,|\nabla v| +C\int\psi+\int\psi\langle f^*\omega_a,\ddc F\rangle_\omega.
 \end{aligned}
\]
To estimate the last term, note that at the normal point
\[\nabla_k(f^*\omega_a)_{i\bar j}=h_{ik}\overline{h_j},\qquad
 \nabla_{\bar k}(f^*\omega_a)_{i\bar j}=h_i\overline{h_{jk}}.
\]
Hence
\[
|\nabla(f^*\omega_a)|_\omega\leq Cu^{1/2}|\nabla^{1,0}_\omega(f^*a)|_\omega.
\]
Integration by parts,  and the bounds
$u+|\nabla F|_\omega\leq C$ from Lemma~\ref{lem:gradient} and
the trace estimate in base  give
\[
 \begin{aligned}
\left|\int\psi\langle f^*\omega_a,\ddc F\rangle_\omega\right|
 &\leq C\int\bigl(u|\nabla\psi|+\psi|\nabla(f^*\omega_a)|\bigr)|\nabla F|\\
&\leq\frac12\int\psi|\nabla^{1,0}_\omega(f^*a)|^2+C\int(|\nabla\psi|+\psi).
 \end{aligned}
\]
Here the square is absorbed using
\[
C\psi|\nabla^{1,0}_\omega(f^*a)|\leq\frac12\psi|\nabla^{1,0}_\omega(f^*a)|^2+C\psi.
\]
Since
\[
|\nabla\psi|\leq2\xi v^{p-1}|\nabla\xi|+(p-1)\xi^2v^{p-2}|\nabla v|,
\]
the remaining terms satisfy
\[
 \begin{aligned}
 2\xi v^{p-1}|\nabla\xi|\,|\nabla v|&\leq\frac{p-1}{4}\xi^2v^{p-2}|\nabla v|^2+\frac{C}{p-1}v^p|\nabla\xi|^2,\\
 C\xi v^{p-1}|\nabla\xi|&\leq Cv^p|\nabla\xi|^2+C\xi^2v^{p-2},\\
 C(p-1)\xi^2v^{p-2}|\nabla v|&\leq\frac{p-1}{4}\xi^2v^{p-2}|\nabla v|^2+C(p-1)\xi^2v^{p-2}.
 \end{aligned}
\]
After absorption, the cutoff bound and $0\leq v\leq C$ imply
\begin{equation}\label{eq:sharp-schwarz-energy}
\int|\nabla(\xi v^{p/2})|^2+\int\xi^2v^p
 \leq Cp^2\bigl(1+(t-s)^{-2}\bigr)\int_{f^{-1}(B_t)}v^{p-2}.
\end{equation}
The uniform Sobolev inequality therefore yields
\[
 \|v\|_{L^{pq}(f^{-1}(B_s))}^p
 \leq Cp^2\bigl(1+(t-s)^{-2}\bigr)\|v\|_{L^{p-2}(f^{-1}(B_t))}^{p-2}.
\]
Then Moser iteration implies
\[
\|v\|_{L^\infty(f^{-1}(B_r))}
 \leq C\|v\|_{L^2(f^{-1}(B_{2r}),\mu_\varepsilon)}^{(q-1)/(2q-1)}
\longrightarrow0.
\]

Write $u_\varepsilon(a),u_0(a)$ to denote the covector dependence. Since $\tr_{\omega_\varepsilon}\chi\leq C$, for $|a|,|b|\leq1$,
\[
 |u_\varepsilon(a)-u_\varepsilon(b)|
 +f^*|u_0(a)-u_0(b)|\leq C|a-b|.
\]
For a finite $\delta$-net of the Euclidean unit covectors, the convergence of $v_+\to 0$ is uniform. Hence \[
\limsup_{\varepsilon\to0}
 \sup_{\substack{x\in f^{-1}(B_r)\\ |a|=1}}
       \bigl(u_\varepsilon(a)-f^*u_0(a)\bigr)
 \leq C\delta.
\]
Let $\delta\downarrow0$. Since $u_0(a)\geq c_r|a|^2$ on
$\overline{B_r}$, homogeneity gives
\[
|f^*\alpha|_{\omega_\varepsilon}^2\leq(1+o_r(1))f^*|\alpha|_{\omega_{\rm can}}^2
 \quad\text{on }f^{-1}(B_r),\qquad \alpha\in T^{*1,0}Y^\circ.
\]
For $Z\in T^{1,0}X$, duality then gives
\[
 |df(Z)|_{\omega_{\rm can}}
 =\sup_{|\alpha|_{\omega_{\rm can}}=1}|f^*\alpha(Z)|
 \leq(1+o_r(1))^{1/2}|Z|_{\omega_\varepsilon}.
\]
Finally choose
$K\subset\bigcup_{j=1}^N B_{r_j}$ with
$\overline{B_{2r_j}}\Subset Y^\circ$.
The maximum of the finitely many errors $o_{r_j}(1)$ tends to zero.
The preceding inequality is therefore equivalent to the desired sharp Schwarz-type estimate on
$f^{-1}(K)$.
\end{proof}

\section{Fiber collapsing and convergence away from singular fibers}
The goal of this section is to prove the metric convergence on the regular part of $X$. 
In the first four lemmas assume $d>0$. All estimates are over a fixed
$K\Subset Y^\circ$, enlarged when necessary to $K'\Subset Y^\circ$
with $K\Subset\operatorname{int}K'$. The smooth proper family
$f^{-1}(K')\to K'$ has uniform Sobolev, Poincar\'e, and Green-function
bounds for $\mathchanged{\omega_y}=\mathchanged{\omega}|_{X_y}$.

\begin{lemma}
    For any $K\Subset Y^\circ$, $\operatorname{osc}_{X_y}\varphi_\varepsilon\leq C_K\varepsilon$, for all $y\in K$.
\end{lemma}
\begin{proof}
    Since $\chi|_{X_y}=0$, one has \[
    \varepsilon^{-1}\omega_\varepsilon|_{X_y}=\mathchanged{\omega_y}+\ddc_V v_{\varepsilon,y}, \quad h_{\varepsilon,y}:=\frac{(\varepsilon^{-1}\omega_\varepsilon|_{X_y})^d}{\mathchanged{\omega_y}^d},
    \]
    where $v_{\varepsilon,y}:=\varepsilon^{-1}\left(\varphi_\varepsilon|_{X_y}-V_F^{-1}\int_{X_y}\varphi_\varepsilon\mathchanged{\omega_y}^d\right),$ and $\ddc_V$ means the differential on the fiber(vertical) direction. Moreover, \[
    \chi^m\wedge\omega_\varepsilon^d=\varepsilon^dh_{\varepsilon,y}\chi^m\wedge\mathchanged{\omega}^d.
    \]
$0\leq \chi\leq \tr_{\omega_\varepsilon}\chi\omega_\varepsilon\leq C\omega_\varepsilon$, $\chi^m\wedge\omega_\varepsilon^d\leq C^m\omega_\varepsilon^n$. Note that $V_\varepsilon\leq C\varepsilon^d$ and $F_\varepsilon\leq C$. Over a compact $f^{-1}(K)\Subset X^\circ$, \[
\omega_\varepsilon^n\leq C\varepsilon^d\Omega\leq C_K\varepsilon^d\chi^m\wedge\mathchanged{\omega}^d.
\]
Hence \[
0<h_{\varepsilon,y}=\frac{\chi^m\wedge\omega_\varepsilon^d}{\varepsilon^d\chi^m\wedge\mathchanged{\omega}^d}\leq \frac{C^m\omega_\varepsilon^n}{\varepsilon^d\chi^m\wedge\mathchanged{\omega}^d}\leq C_K.
\]
Standard $L^\infty$-estimate for complex Monge-Amp\`ere equation gives the desired estimate for $v_{\varepsilon,y}$.    
\end{proof}

Next we show that the normalized fiber potential $v_{\varepsilon,y}:=\varepsilon^{-1}\left(\varphi_\varepsilon|_{X_y}-V_F^{-1}\int_{X_y}\varphi_\varepsilon\mathchanged{\omega_y}^d\right)$ indeed converges the the potential of $\omega_{\rm SRF}.$
\begin{lemma}
    $\sup_{y\in K}\|v_{\varepsilon,y}-\rho_y\|_{L^\infty}\to 0$, where $\mathchanged{\omega}+\ddc\rho_y$ is the unique Ricci flat metric in class $[\mathchanged{\omega}|_y]$ with $\int_{X_y}\rho_y\mathchanged{\omega_y}^d=0$.
\end{lemma}
\begin{proof}
    By Lemma~\ref{lem:sharp-schwarz-estimate}, we can choose $0<\delta_\varepsilon\to 0$ such that $\omega_\varepsilon-(1-\delta_\varepsilon)f^*\omega_{\rm can}\geq 0$. Thus \[\begin{aligned}
        \omega_\varepsilon^n\geq&\binom{n}{m}(1-\delta_\varepsilon)^m(f^*\ocan)^m\wedge(\oep-(1-\delta_\varepsilon)f^*\ocan)^d\\
        =&\binom{n}{m}(1-\delta_\varepsilon)^m(f^*\ocan)^{m}\oep^d.
    \end{aligned}
    \]
    Let $h_{{\rm RF},y}=\omega_{{\rm SRF},y}^d/\mathchanged{\omega_y}^d$. Then \[
    \frac{h_{\varepsilon,y}}{h_{{\rm RF},y}}=\frac{(f^*\ocan)^m\wedge\omega_\varepsilon
    ^d}{\varepsilon^d(f^*\ocan)^m\wedge\osrf^d}\leq \frac{\omega_\varepsilon^n}{(1-\delta_\varepsilon)^m\binom{n}{m}\varepsilon^d(f^*\ocan)^m\wedge\osrf^d}=1+o(1).
    \]
    Here we used that $F_\varepsilon\to f^*\pcan$ in $L^\infty$ away from singlar fiber(Lemma~\ref{lem:Linfty-density}). 
    Consequently, \[
    0<h_{\varepsilon,y}\leq (1+\eta_{\varepsilon})h_{{\rm RF},y}, \quad \eta_{\varepsilon}\to 0.
    \]
    From stability estimate for complex Monge-Amp\`ere equation, we get the desired $L^\infty$-closeness of the potentials. To see the convergence is uniform on $K$, we sketch a proof here. Set $a=v_{\varepsilon,y}-\rho_y$, and
$\omega_{\varepsilon,y}'=\omega_{{\rm SRF},y}+dd^ca$.
Both forms are positive, $\|a\|_\infty\leq C_K$, and
$(\omega_{\varepsilon,y}')^d/\omega_{{\rm SRF},y}^d\leq C_K$.
The smooth reference metrics $\omega_{{\rm SRF},y}$ over $K'$ have common
Sobolev and Poincar\'e constants and are uniformly equivalent to
$\mathchanged{\omega_y}$. Since $\int a\,\mathchanged{\omega_y}^d=0$, integration by
parts gives
\[
 \begin{split}
\int_{X_y}\sqrt{-1}\partial a\wedge\bar\partial a\wedge\omega_{{\rm SRF},y}^{d-1}
&\leq\int_{X_y}\sqrt{-1}\partial a\wedge\bar\partial a
\wedge\sum_{j=0}^{d-1}(\omega_{\varepsilon,y}')^j\wedge\omega_{{\rm SRF},y}^{d-1-j}\\
&=-\int_{X_y}a((\omega_{\varepsilon,y}')^d-\omega_{{\rm SRF},y}^d)
\leq C_K\eta_\varepsilon,\\
\|a\|_{L^2(\omega_{{\rm SRF},y}^d)}&\leq C_K\eta_\varepsilon^{1/2}.
\end{split}
\]
Let $u=a_+$, from difference of two equations we get
\[
 \begin{split}
 (p-1)\int u^{p-2}\sqrt{-1}\partial u\wedge\bar\partial u\wedge\omega_{{\rm SRF},y}^{d-1}
 &\leq C_K\int u^{p-1}\omega_{{\rm SRF},y}^d,\\
 \|u\|_{L^{pq}(\omega_{{\rm SRF},y}^d)}^p
 &\leq C_Kp^2\int u^{p-1}\omega_{{\rm SRF},y}^d .
 \end{split}
\]
Here $q>1$ is a common fiber Sobolev exponent, and the zero-order
Sobolev term is absorbed using $0\leq u\leq C_K$. Moser iteration gives
\[
 \|u\|_{L^\infty}\leq C_K\|u\|_{L^2}^\beta,\qquad\beta=\lim_{k\to\infty}\frac{2q^k}{p_k} =\frac{2}{2+q/(q-1)}>0.
\]
Applying this to both signs proves
$\|a\|_{L^\infty}\leq C_K\eta_\varepsilon^{\beta/2}\to0$.

\end{proof}

\begin{lemma}
The form
\[
 T_\varepsilon=
 \frac{f_*(\omega_\varepsilon^{d+1}-\varepsilon^{d+1}\omega_{\rm SRF}^{d+1})}{(d+1)\varepsilon^dV_F}
\]
is smooth and closed on $Y^\circ$. For any compact $K\Subset Y^\circ$,  $T_\varepsilon\to \ocan$ uniformly on $K$. 
\end{lemma}

\begin{proof} 
As before, choose $0<\delta_\varepsilon\to 0$ such that $\gamma_\varepsilon:=\oep-(1-\delta_\varepsilon)f^*\ocan\geq 0$. Then $R_\varepsilon:=f_*(\gamma_\varepsilon^{d+1})/((d+1)\varepsilon^dV_F)\geq 0$.  \[
\omega_\varepsilon^{d+1}-\gamma_\varepsilon^{d+1}=(1-\delta_\varepsilon)f^*\ocan\sum_{i=0}^{d}\omega_\varepsilon^i\gamma_\varepsilon^{d-i}.
\]
Only products in the fiber direction contribute to the fiber integration, hence \[
f_*(\sum_{i=0}^{d}\omega_\varepsilon^i\gamma_\varepsilon^{d-i})=\int_{X_y}(\omega_{\varepsilon,y})^d=(d+1)V_F\varepsilon^d.
\]
As a result, $T_\varepsilon=(1-\delta_\varepsilon)\ocan+R_\varepsilon-\frac{\varepsilon f_*(\osrf^{d+1})}{(d+1)V_F}$.
On the other hand, \[\begin{aligned}
f_*\omega_\varepsilon^n\geq &f_*\left(\binom{n}{d+1}\gamma_\varepsilon^{d+1}(1-\delta_\varepsilon)^{m-1}f^*\ocan^{m-1}+\binom{n}{m}\gamma_\varepsilon^d(1-\delta_\varepsilon)^mf^*\ocan^m\right)\\
=&\binom{n}{d}\varepsilon^dV_F[(1-\delta_\varepsilon)^{m-1}mR_\varepsilon\wedge\ocan^{m-1}+(1-\delta_\varepsilon)^m\ocan^m]
\end{aligned}
\]
Hence
\[
\frac{f_*\omega_\varepsilon^n}{\binom{n}{m}\varepsilon^d V_F\ocan^m}\geq (1-\delta_\varepsilon)^{m-1}\tr_{\ocan}R_\varepsilon+(1-\delta_\varepsilon)^m.
\]
Since $f_*\omega_\varepsilon^n\to V_{\varepsilon}f^*\ocan^m$ by Lemma~\ref{lem:Linfty-density}, we get \[
0\leq |R_\varepsilon|_{\ocan}\leq \tr_{\ocan}R_\varepsilon\to 0.
\]
This proves $T_\varepsilon\to \ocan$ uniformly away from singular fibers. 
\end{proof}

\begin{lemma}\label{lem:r-varepsilon}
 There are smooth function $r_\varepsilon$ on $X^\circ$ such that for any compact $K\Subset Y^\circ$\[
    \omega_\varepsilon=f^*T_\varepsilon+\varepsilon\osrf+\varepsilon\ddc r_\varepsilon \text{ on }X^\circ, \qquad \|r_\varepsilon\|_{L^\infty(f^{-1}(K))}\to 0. 
    \]
\end{lemma}
\begin{proof}
    For a potential $\varphi$, recall that the energy is \[
    E_y(\varphi):=\frac{1}{(d+1)V_F}\sum_{j=0}^d\int_{X_y}\varphi(\mathchanged{\omega_y}+\ddc\varphi)^j\wedge\mathchanged{\omega_y}^{d-j}.
    \]
    Here and in the following $\ddc$ is taken with respect to the full directions. 
    It follows that \[
    f_*((\mathchanged{\omega}+\ddc\varphi)^{d+1}-\mathchanged{\omega}^{d+1})=(d+1)vV_F\ddc E_y(\varphi)
    \]
    We can write $\oep=f^*(\chi_Y+\ddc \bar{\varphi}_{\varepsilon,y})+\varepsilon(\mathchanged{\omega}+\ddc v_{\varepsilon,y})$, where $\bar{\varphi}_{\varepsilon,y}=V_F^{-1}\int_{X_y}\varphi_{\varepsilon}\mathchanged{\omega_y}^d$.
    Hence \[
 \frac{f_*(\omega_\varepsilon^{d+1})}
      {(d+1)\varepsilon^dV_F}
 =\chi_Y+dd^c\bar{\varphi}_{\varepsilon,y}
   +\frac{\varepsilon f_*((\mathchanged{\omega}+dd^c v_{\varepsilon,y})^{d+1})}
          {(d+1)V_F}.
    \]
    By the definition of $T_\varepsilon$, we get \[
    T_\varepsilon=\chi_Y+\ddc \bar{\varphi}_{\varepsilon,y}+\varepsilon\ddc(E_y(v_{\varepsilon,y})-E_{y}(\rho)).
    \]
    Note that \[
    |E_y(v_{\varepsilon,y})-E_y(\rho)|\leq \|v_{\varepsilon,y}-\rho\|_{L^\infty}\to 0.
    \]
    Hence let $r_\varepsilon=v_{\varepsilon,y}-\rho-f^*(E_y(v_{\varepsilon,y})-E_y(\rho))$, we get the desired representation and the estimate. 
\end{proof}

The following $\ddc$-lemma with estimate is standard.

\begin{lemma}\label{lem:local-potential}

Let $\mathcal P\subset\mathbb C^N$ be a fixed polydisc. A smooth closed real
$(1,1)$-form $\alpha$ on a neighborhood of $\overline{\mathcal P}$,
has a real potential $b\in C^\infty(\mathcal P)$ such that
\begin{equation}
dd^c b=\alpha,\qquad
\|b\|_{L^\infty(\mathcal P)}\leq C_{\mathcal P}\|\alpha\|_{L^\infty(\mathcal P)}.
\end{equation}
Here the norms of forms are Euclidean tensor norms.
\end{lemma}

Now we can prove the metric convegence by a blow-up argument. We will use the following theorem of Cheng-Xu on the interior $W^{2,p}$-estimate for complex Monge-Amp\`ere equations, regardless of the boundary behavior. The addition assumption is the closeness to a smooth stricly psh function. \begin{theorem}[Cheng-Xu, Theorem 1.4\cite{CX23}]\label{thm:Cheng-Xu.}
    Suppose on $B_1\subset\mathbb C^n$ that
\[
 C_0^{-1}\omega_{\rm E}\leq dd^cw_0\leq C_0\omega_{\rm E},
 \qquad\|D^3w_0\|_\infty\leq C_0.
\]
For every $1<P<\infty$, there are tolerances $\delta_0(C_0,n)>0$
and $\delta_{\rm MA}(n,P)>0$ such that
\[
 \begin{gathered}
 \phi\in C^0\cap\PSH(B_1),\qquad\|\phi-w_0\|_\infty\leq\delta_0,\\
 (1-\delta_{\rm MA})\omega_{\rm E}^n\leq(dd^c\phi)^n\leq(1+\delta_{\rm MA})\omega_{\rm E}^n
 \end{gathered}
 \quad\Longrightarrow\quad
 \|\phi\|_{W^{2,P}(B_{1/2})}\leq C.
\]
\end{theorem} 

\begin{lemma}\label{lem:metric-convergence}For every compact $K\Subset Y^\circ$, \begin{enumerate}
    \item $\|\oep-
    (f^*\ocan+\varepsilon\osrf)\|_{f^*\ocan+\varepsilon\mathchanged{\omega}}\to 0$ uniformly on $f^{-1}(K)$.
    \item $\sup_{y\in K}\diam(X_y,\omega_\varepsilon|_{X_y})\leq C_K\sqrt{\varepsilon}.$
\end{enumerate}
\end{lemma}
\begin{proof}
If $d=0$, $f:X^\circ\to Y^\circ$ is biholomorphic.
The sharp Schwarz estimate gives
$\omega_\varepsilon\geq(1-o(1))f^*\omega_{\rm can}$, while
Lemma~\ref{lem:Linfty-density} gives
$\omega_\varepsilon^n/(f^*\omega_{\rm can})^n\to1$.
Every eigenvalue is bounded below by $1-o(1)$ and their product
tends to $1$; hence every eigenvalue tends to $1$.
This proves (1), with the convention $\omega_{\rm SRF}=\mathchanged{\omega}$.
The fibers are points, so (2) is immediate. Assume henceforth $d>0$.

Suppose $\varepsilon_j\to 0$, and choose arbitrary $x_j\in f^{-1}(K)$. After passing to subsequence, $x_j\to x_0\in f^{-1}(K)$. We fix a submersion coordinates $(y,z)\in \C^m\times \C^d$ around $x_0$, such that $f(y,z)=y$. Suppose $x_j=(y_j,z_j)$. Define on a fixed polydisc centered at zero\[
    D_j(w,\zeta)=(y_j+\sqrt{\varepsilon_j}w,z_j+\zeta), \quad \tom_j=\varepsilon_j^{-1}D_j^*\omega_{\varepsilon_j}
    \]
By Lemma~\ref{lem:r-varepsilon}, \[
\tom_j=\varepsilon_j^{-1}D_j^*(f^*T_{\varepsilon_j}+\varepsilon_j\osrf)+\ddc(r_{\varepsilon_j}\circ D_j).
\]
The reference forms converges \[
\varepsilon_j^{-1}D_j^*(f^*T_{\varepsilon_j}+\varepsilon_j\osrf)\to \omega_\infty:=\pr_1^*\ocan(y_0)+\pr_2^*\osrf|_{X_{y_0}}\quad \text{ in } L^\infty,
\]
where $\pr_1$ is the projection to $y$ and $\pr_2$ is the projection to $z$. 

The previous lemma provides $b_j$ such that \[
\ddc b_j=\varepsilon_{j}^{-1}D_j^*(f^*T_{\varepsilon_j}+\varepsilon_j\osrf)-\omega_\infty, \quad \|b_j\|_{L^\infty}\to 0.
\]
Fix $w_0$ with $\ddc w_0=\omega_\infty$. Then \[
\phi_j:=w_0+b_j+r_{\varepsilon_j}\circ D_j, \qquad \ddc\phi_j=\tom_j, \qquad \|\phi_j-w_0\|_{L^\infty}\to 0.
\]
Let $\omega_E=\ddc(|w|^2+|\zeta|^2)$, then we have the volume density convergence \[
\frac{(\ddc\phi_j)^n}{\omega_E^n}=\frac{\varepsilon_j^{-n}D_j^*(\omega_{\varepsilon_j}^n)}{\omega_E^n}\to \frac{\omega_\infty^n}{\omega_E^n}, \quad \text{ in } L^\infty.
\]
Choose $p>3n(n-1)$ and $P>\max\{2n,p,(n-1)p\}$. After a linear change, we can assume $\ddc w_0=\omega_E(0)$. Choose a fixed $r_*>0$ such taht \[
\left|\frac{(\ddc w_0)^n}{\omega_E^n}-1\right|<\frac{1}{2}\delta_{\MA}(n,P) \text{ on }B_{r_*}.
\]
Pull back by dilation $z\mapsto r_*z$ and divide the potentials by
$r_*^2$, retaining their names. The rescaled reference $w_0$ is fixed,
so its ellipticity and third derivatives are bounded by some $C_0$.
Since $r_*$ is fixed, for all
sufficiently large $j$,
\[
 \begin{gathered}
 \|\phi_j-w_0\|_\infty\leq\delta_0(C_0,n),\\
 (1-\delta_{\rm MA}(n,P))\omega_{\rm E}^n
 \leq(dd^c\phi_j)^n
 \leq(1+\delta_{\rm MA}(n,P))\omega_{\rm E}^n.
 \end{gathered}
\]
By Theorem~\ref{thm:Cheng-Xu.}, we get $\|\phi_j\|_{W^{2,P}(B_\frac12)}\leq C.$ Note that for any positive $(1,1)$-form $\alpha$, \[
\tr_\alpha\omega_E\leq (\tr_{\omega_E}\alpha)^{n-1}\frac{\omega_E^n}{\alpha^n}.
\]
Let $\alpha=\ddc\phi_j$, then we get\[
\|\tr_{\omega_E}\ddc\phi_j\|_{L^P(B_\frac12)}+\|\tr_{\ddc\phi_j}\omega_E\|_{L^{P/(n-1)}(B_\frac12)}\leq C.
\]
In particular both traces have bounded $L^p$ norm. 
We also have the cscK equation $R(\ddc\phi_j)=r_*^2\varepsilon_j\underline{R}_{\varepsilon_j}$. So $|R(\ddc\phi_j)|+|\phi_j|\leq C$.
Then we can apply Chen-Cheng's local estimate \cite[Section 6]{CC17}, on $B_\frac14$\[
C^{-1}\omega_E\leq \ddc\phi_j\leq C\omega_E, \qquad \|\phi_j\|_{C^k}\leq C_k,
\]for any $k>0$.
Hence we get the smooth convergence $\tom_j\to \omega_\infty$.

The model $\varepsilon_j^{-1}D_j^*(f^*\ocan
+\varepsilon_j\osrf)$ also converges to $\omega_\infty$.
Moreover, $\varepsilon_j^{-1}D_j^*(f^*\ocan
+\varepsilon_j\mathchanged{\omega})$ has a positive smooth limit. Invariance of
relative form norms under pullback and common scaling gives
\begin{equation}
 \begin{split}
&\left|\omega_{\varepsilon_j}-(f^*\ocan+\varepsilon_j\osrf)\right|_{f^*\ocan+\varepsilon_j\mathchanged{\omega}}(x_j)\\
 &\quad=\left|\tom_j-\varepsilon_j^{-1}D_j^*(f^*\ocan+\varepsilon_j\osrf)\right|_{\varepsilon_j^{-1}D_j^*(f^*\ocan+\varepsilon_j\mathchanged{\omega})}(0)\longrightarrow0.
 \end{split}
\end{equation}
The centers $x_j$ were arbitrary, so (i) is proved.

Restriction $\omega_\varepsilon$ to a fiber, we have 
$\varepsilon^{-1}\omega_\varepsilon|_{X_y}
\leq C_K\omega_{{\rm SRF},y}$. The fixed compact smooth family of
Ricci-flat fibers has uniformly bounded diameters. This proves (ii). 

\end{proof}

\section{Gromov-Hausdorff convergence from Lee--Tosatti--Zhang}
The main content of this section is a summary of the work of Sz\'ekelyhidi on the RCD structure on the canonical space and the argument from Lee--Tosatti--Zhang on the Gromov-Hausdorff convergence of the K\"ahler--Ricci flow. Those enable us to reduce the convergence to a key one-side estimate of the distance functions, Proposition~\ref{prop:lower-distance}.

One side of the comparision of the distance function is easier. It follows from the argument of Lee--Tosatti--Zhang \cite[Section 4.1]{LTZ26}, we summarize it here:
\begin{lemma}\label{lem:ltz-distance}
For $0<m\leq n$:
\begin{enumerate}
\item(\cite[Lemma~4.1]{LTZ26})
For every $\delta>0$, there is a compact $K\Subset Y^\circ$,
independent of $\varepsilon$, such that, for all sufficiently small
$\varepsilon>0$,
\[
 \sup_{y\in Y}d_{\rm can}(y,K)<\delta,\qquad
 \sup_{x\in X}d_{\omega_\varepsilon}(x,f^{-1}(K))<\delta.
\]

\item(\cite[equation~(4.14), using Lemma~4.2]{LTZ26})\label{item:ltz-regular-upper}
For every compact $K\Subset Y^\circ$,
\begin{equation}
 \sup_{x,x'\in f^{-1}(K)}
 \bigl(d_{\omega_\varepsilon}(x,x')
       -d_{\rm can}(f(x),f(x'))\bigr)_+\longrightarrow0.
\end{equation}
\end{enumerate}
\end{lemma}
\begin{proof}
The cited arguments are metric arguments and apply to the present
family with the following inputs. For each $K\Subset Y^\circ$,
\begin{equation}\label{eq:ltz-local-input}
 \begin{gathered}
 \|\omega_\varepsilon-f^*\omega_{\rm can}\|
       _{C^0(f^{-1}(K),\mathchanged{\omega})}\longrightarrow0,\\
 \sup_{y\in K}\diam(X_y,\omega_\varepsilon|_{X_y})
       \longrightarrow0.
 \end{gathered}
\end{equation}
Theorem~\ref{thm:bounds} gives, with constants independent of
$\varepsilon$,
\begin{equation}\label{eq:ltz-volume-input}
 \mu_\varepsilon\leq C\Omega,\qquad
 \mu_\varepsilon(B_{\omega_\varepsilon}(x,r))\geq C^{-1}r^{a_{\rm vol}}
 \quad(x\in X,\ 0<r<1).
\end{equation}
Since $f^{-1}(D)$ is proper analytic, $\Omega(f^{-1}(D))=0$.
Consequently, for a compact exhaustion $K_j\Subset Y^\circ$,
\[
 \sup_\varepsilon\mu_\varepsilon(X\setminus f^{-1}(K_j))
 \leq C\int_{X\setminus f^{-1}(K_j)}\Omega\longrightarrow0.
\]
Together with the theorem of Szekelyhidi~\ref{thm:completion}, these are the inputs to
\cite[Lemma~4.1]{LTZ26}. The regular path construction in
\cite[Lemma~4.2]{LTZ26} and \eqref{eq:ltz-local-input} give
\cite[equation~(4.14), in the proof of Lemma~4.3]{LTZ26}.
No evolution equation is used in these steps.
\end{proof}
The key to the global Gromov-Hausdorff convergence is the following reverse inequality. 
\begin{proposition}\label{prop:lower-distance}
For $0<m\leq n$,
\[
\sup_{x,x'\in X}
\big(d_{\rm can}(f(x),f(x'))-d_{\omega_\varepsilon}(x,x')\big)_+
\longrightarrow0 .
\]
\end{proposition}
We postpone the proof of this key proposition to next section, and show how the Gromov-Hausdorff convergence can be proved from it. 

\begin{proposition}\label{prop:distortion}\label{prop:upper-distance}
For $0<m\leq n$:
\begin{enumerate}
\item\label{item:distortion-uniform}
\[
 \sup_{x,x'\in X}\left|d_{\omega_\varepsilon}(x,x') -d_{\rm can}(f(x),f(x'))\right|\longrightarrow0.
\]
\item\label{item:distortion-gh}
\[
 d_{\rm GH}\bigl((X,d_{\omega_\varepsilon}),(Y,d_{\rm can})\bigr)\longrightarrow0.
\]
\end{enumerate}
\end{proposition}

\begin{proof}
Fix $\delta>0$, and choose $K\Subset Y^\circ$ by
Lemma~\ref{lem:ltz-distance}\textup{(i)}. By
Lemma~\ref{lem:ltz-distance}\textup{(ii)} and
Proposition~\ref{prop:lower-distance}, for all sufficiently small
$\varepsilon>0$,
\[
 \begin{aligned}
 d_{\omega_\varepsilon}(z,z')
 &\leq d_{\rm can}(f(z),f(z'))+\delta
 &&(z,z'\in f^{-1}(K)),\\
 d_{\rm can}(f(x),f(x'))
 &\leq d_{\omega_\varepsilon}(x,x')+\delta
 &&(x,x'\in X).
 \end{aligned}
\]
For arbitrary $x,x'\in X$, choose $z,z'\in f^{-1}(K)$ with
\[
d_{\omega_\varepsilon}(x,z)<\delta,\qquad
d_{\omega_\varepsilon}(x',z')<\delta.
\]
The lower distance estimate then gives
\[
 d_{\rm can}(f(x),f(z))<2\delta,\qquad
 d_{\rm can}(f(x'),f(z'))<2\delta.
\]
Hence
\[
 \begin{aligned}
d_{\omega_\varepsilon}(x,x')&\leq2\delta+d_{\omega_\varepsilon}(z,z')\\
&\leq3\delta+d_{\rm can}(f(z),f(z'))\\
&\leq7\delta+d_{\rm can}(f(x),f(x')).
 \end{aligned}
\]
First $\varepsilon\to0$, then $\delta\downarrow0$, yields
\begin{equation}\label{eq:u9}
 \sup_{x,x'\in X}\bigl(d_{\omega_\varepsilon}(x,x')-d_{\rm can}(f(x),f(x'))\bigr)_+\longrightarrow0.
\end{equation}
Combining this with Proposition~\ref{prop:lower-distance} proves
\textup{(i)}. This is the metric reduction of
\cite[Lemma~4.3]{LTZ26}, with the uniform lower estimate also
controlling the endpoints outside $f^{-1}(K)$.

For \textup{(ii)}, the surjectivity of $f$ makes
$\mathcal R_f=\{(x,f(x)):x\in X\}$ a correspondence. Thus
\[
 \begin{aligned}
 d_{\rm GH}\bigl((X,d_{\omega_\varepsilon}),(Y,d_{\rm can})\bigr)
 &\leq\tfrac12\operatorname{dis}(\mathcal R_f)\\
 &=\tfrac12\sup_{x,x'\in X}
       |d_{\omega_\varepsilon}(x,x')-d_{\rm can}(f(x),f(x'))|
 \longrightarrow0.
 \end{aligned}\qedhere
\]
\end{proof}

\section{Proof of Proposition~\ref{prop:lower-distance}}

The key proposition can be reduced to the following lemma on Lipschitz approximation on the lift:\begin{lemma}\label{lem:lip-lift-approximation}
    For any $v\in \Lip(Y,\dcan)$ with $\Lip_{\dcan}(v)\leq 1$, and any $\delta>0$, there exits smooth functions $v_\varepsilon$ such that $\limsup_{\varepsilon\to 0}\|v_{\varepsilon}-f^*v\|_{L^\infty}\leq \delta$ and $\Lip_{d_{\oep}}(v_\varepsilon)\leq 1+\delta$.
 \end{lemma}
\begin{proof}[Lemma~\ref{lem:lip-lift-approximation} implies Proposition~\ref{prop:lower-distance}]

Fix $\delta>0$ and a finite $\delta$-net $\{z_1,\ldots,z_N\}$ in
$(Y,d_{\rm can})$. For $v_j(z)=d_{\rm can}(z,z_j)$,
\[
 \Lip(v_j)\leq1,\qquad
 d_{\rm can}(z,z')\leq2\delta+\max_j|v_j(z)-v_j(z')|.
\]
Choose $v_{\varepsilon,j}$ from the preceding lemma, 
and set
\[
 \begin{aligned}
 a_\varepsilon&=\max_j
       \bigl(\|v_{\varepsilon,j}-f^*v_j\|_\infty-\delta\bigr)_+\to0,\\
 b_\varepsilon&=\max_j
       \bigl(\Lip_{\omega_\varepsilon}(v_{\varepsilon,j})-1-\delta\bigr)_+\to0.
 \end{aligned}
\]
For all $x,x'\in X$,
\[
 \begin{aligned}
 d_{\rm can}(f(x),f(x'))
 &\leq2\delta+\max_j|v_j(f(x))-v_j(f(x'))|\\
 &\leq4\delta+2a_\varepsilon
                +\max_j|v_{\varepsilon,j}(x)-v_{\varepsilon,j}(x')|\\
 &\leq4\delta+2a_\varepsilon
                +(1+\delta+b_\varepsilon)d_{\omega_\varepsilon}(x,x').
 \end{aligned}
\]
Using $\diam(X,\omega_\varepsilon)\leq C$ from
Theorem~\ref{thm:bounds},
\[
 \limsup_{\varepsilon\to0}\sup_{x,x'\in X}
 \bigl(d_{\rm can}(f(x),f(x'))-d_{\omega_\varepsilon}(x,x')\bigr)_+
 \leq(4+C)\delta.
\]
Let $\delta\downarrow0$.
\end{proof}

We aim to prove Lemma~\ref{lem:lip-lift-approximation}.
Before moving to the proof, we introduce a zero-capacity lemma due to Sturm, which turned out to be important in the study of singular K\"ahler space, see \cite[Lemma 3.7]{Song14}\begin{lemma}\label{lem:zero-capacity}
    There are $\xi_j\in C_c^\infty(Y^\circ)$ with
$0\leq\xi_j\leq1$, $\xi_j\to1$ locally on $Y^\circ$, and
\[
 \int_Y\bigl((1-\xi_j)^2+|\nabla\xi_j|_{\omega_{\rm can}}^2\bigr)d\nu_{\rm can}\to0.
\]
\end{lemma}
The cited construction applies to the analytic set
$D=Y\setminus Y^\circ$: the current $\omega_{\rm can}$ has bounded
local potentials and is smooth on $Y^\circ$. The comparison-flow
trace bound and its local limit give
$\omega_{\rm can}\geq C^{-1}\chi_Y$ on $Y^\circ$; bounded local
potentials extend this inequality across $D$.
Thus the K\"ahler-current hypothesis of the cited lemma holds.

In this subsection, gradients and integrals are taken on the smooth manifold
$(Y^\circ,\omega_{\rm can})$, with measure $\nu_{\rm can}$.
The complement $D$ has measure and capacity zero. Thus we are in the
almost smooth setting of \cite[Definition~5]{S25RCD}, up to the constant
normalization of volume; the smooth open set is $Y^\circ$.
In particular,
\[
 W^{1,2}(Y,
 \ocan,\nu_{\rm can})
 =\overline{C^\infty_c(Y^\circ)}^{\|\cdot\|_{W^{1,2}}},
 \qquad
 \|u\|_{W^{1,2}}^2
 =\int_{Y^\circ}(u^2+|\nabla u|_{\omega_{\rm can}}^2)d\nu_{\rm can}.
\]
Indeed, for $u\in\Lip(Y)$,
\[
 \begin{aligned}
 \|(1-\xi_j)u\|_{L^2}&\leq\|u\|_{\infty}\|1-\xi_j\|_{L^2}\longrightarrow0,\\
 \|d((1-\xi_j)u)\|_{L^2}
 &\leq\|(1-\xi_j)du\|_{L^2}+\|u\|_{L^\infty}\|d\xi_j\|_{L^2}
 \longrightarrow0.
 \end{aligned}
\]
Smooth approximation of $\xi_ju$ in coordinate balls, followed by
density of Lipschitz functions in the metric Sobolev space, proves
the equality. The same cutoff estimate applies to any bounded
$u\in W^{1,2}_{\rm loc}(Y^\circ)$ with $du\in L^2$. in the following, without ambiguity, we write $W^{1,2}$ for $W^{1,2}(Y,\ocan,\nu_{\rm can}),$  write $\omega,\nu$ for $\ocan,\nu_{\rm can}$ and write $\Delta $ for $\Delta_{\ocan}$.

For $u\in W^{1,2}$, and $g\in L^\infty$, $u$ is a said to be a weak solution of $\Delta u=g$ if $\int_{Y^\circ}\langle \nabla u,\nabla v\rangle_{\ocan}d\nu_{\rm can}+\int_{Y^\circ}gvd\nu_{\rm can}=0$ for all $v\in W^{1,2}$.
For $\lambda>0$, $g\in L^\infty$, 
by the standard Lax-Milgram theorem, there is a unique $h\in W^{1,2}$ such that $\Delta h=\lambda h+g$. Moreover, $\|h\|_{L^\infty}\leq \lambda^{-1}\|g\|_{L^\infty}$, and if $g\in C^\infty(Y^\circ)$, then $h\in C^\infty(Y^\circ).$  Indeed, to see the $L^\infty$-bound, just take the test function to be $v=(h-\lambda^{-1}\|g\|_{L^\infty})_+$. 
\begin{lemma}
    \label{lem:Resolvent estimate}
Let $h\in W^{1,2}$ be the unique solution to $\Delta h=\lambda h+g$ for $\lambda>0$, and $g\in L^\infty$. Then $\Lip(h)\leq C_\lambda\|g\|_{L^\infty}.$    
\end{lemma}
\begin{proof}
    Szekelyhidi's theorem shows that the $(Y,\dcan,\nu_{\rm can})$ is a $RCD(-1,2m)$ space. \cite[Theorem 2.33]{HS26}(see also \cite{JKY14}) shows the quantitative Lipschitz regularity of solution of Poisson equations over balls.  Cover $Y$ by finite balls gives \[
    |\nabla h|(x)\leq C(\|h\|_{L^\infty}+\|\Delta h\|_{L^\infty})\leq C_\lambda\|g\|_{L^\infty},
    \] for $\nu_{\rm can}$-almost every $ x$.  Such $h$ has a unique $\Lip$-representative; see the proof of \cite[Proposition 9]{S25RCD}.
\end{proof}

\begin{lemma}\label{lem:sharp-lip}
    For $\lambda>1$ and $g\in \Lip(Y)$. The solution to $\Delta h=\lambda h+g$ satisfies the sharper $\Lip$-bound: \[
    \Lip(h)\leq \frac{1}{\lambda-1/2}\Lip(g).
    \]Moreover, $\|\lambda h+g\|_{L^\infty}\to 0$ as $\lambda\to \infty.$
\end{lemma}
\begin{proof}
    From local regularity, $h\in W^{3,p}_{loc}(Y^\circ)$ for any $p>1$, and hence locally $C^2$. The previous lemma gives $|\nabla h|^2\in L^\infty$. Since $\Ric(\ocan)\geq -\ocan$, we have, on $Y^\circ$, \[
    \Delta|\nabla h|_\omega^2\geq \sum_{i,j}(|h_{ij}|^2+|h_{i\bar j}|^2)+(2\lambda-1)|\nabla h|_\omega^2+2\operatorname{Re}h_ig_{\bar i}\geq (2\lambda-1)|\nabla h|_\omega(|\nabla h|_\omega-L),
    \]where we take $L=(2\lambda-1)^{-1}\sqrt{2}\Lip(g)$. 
    Let $w=|\nabla u|_\omega^2$ and $z=(w-L^2)_+$. Test the Bochner formula with $\xi_j^2z$, where $\xi_j$ is the cut-off with capacity tends to $0$. Integration by parts and Young's inequality give
\[
 \begin{aligned}
\int\xi_j^2|\nabla z|_\omega^2
 +(2\lambda-1)\int\xi_j^2z\sqrt w(\sqrt w-L)
 &\leq-\int\xi_jz\langle \nabla \xi_j,\nabla w\rangle\\
 &\leq\frac14\int\xi_j^2|\nabla z|^2
       +\|z\|_\infty^2\int|\nabla\xi_j|^2.
 \end{aligned}
\]
Absorb the first term on the right. By  Fatou's lemma,
\[
 (\lambda-1)\int z\sqrt w(\sqrt w-L)=0
 \quad\Longrightarrow\quad z=0
 \quad\Longrightarrow\quad
 \Lip(h)\leq\frac{\Lip(g)}{\lambda-1/2}.
\]

For the uniform convergence, test the equation with $\lambda h+g$, we get \[
\|\lambda h+g\|_{L^2}\leq \frac{1}{2\sqrt{\lambda}}\|\nabla g\|_{L^2}\to 0. 
\]
After scaling, assume $\Lip(g)\leq 1$. For $\lambda>2$, $\Lip(\lambda h+g)\leq 3$. Hence if $\|\lambda h+g\|_{L^\infty}\ge a>0$, \[
\|\lambda h+g\|_{L^2}^2\geq \frac{a^2}{4}\inf_{y\in Y}\nu_{\rm can}(B_{\dcan}(y,a/6))>0.
\]A contradiction. 
\end{proof}

\begin{lemma}\label{lem:removable}For $\lambda>0$, $g\in L^\infty(Y)$, and $h\in C^\infty(Y^\circ)\cap L^\infty(Y^\circ)$, if $h$ satisfies $\Delta h=\lambda h+g$ pointwise on $Y^\circ$, then $h\in W^{1,2}$ and is the weak solution to $\Delta h=\lambda h+g$ on $Y$. 
\end{lemma}
\begin{proof}
   Let $\xi_j$ be the sequence of cut-offs with capacity tends to $0$. Choose test function $\xi_j^2h$, then \[
   \begin{aligned}
\int\xi_j^2|\nabla h|^2+\lambda\int\xi_j^2h^2
 &=-\int g\xi_j^2h-2\int\xi_jh\langle \nabla h,\nabla \xi_j\rangle\\
&\leq\|g\|_{L^\infty}\|h\|_{L^\infty}+\frac12\int\xi_j^2|\nabla h|^2+2\|h\|_{L^\infty}^2\int|\nabla \xi_j|^2.
 \end{aligned}
   \]
   Here all integrals use $d\nu_{\rm can}$, and gradients are taken with respect to the metric $\ocan$.
By Fatou's lemma  $\nabla h\in L^2$. It follows that $h\in W^{1,2}$. The weak equation holds globally by taking limit of test functions in $W^{1,2}$. 
\end{proof}

\begin{lemma}\label{lem:approx-lip}
For any $v\in \Lip(Y)$ with $\Lip(v)\leq 1$, and every $\delta>0$, there exists $\lambda>1$ and a smooth function $g$ on $\C P^N$, such that the solution to the equation $\Delta h=\lambda h+g|_Y$ satifies \[
h\in C^\infty(Y^\circ)\cap \Lip(Y), \quad \|h-v\|_{L^\infty}<\delta, \quad \Lip(h)<1+\delta.
\]
\end{lemma}
\begin{proof}By Lemma~\ref{lem:sharp-lip}, we can choose $\lambda>1$ such that the solution $\Delta h_v
=\lambda h_v+v$, satisfies \[
\Lip(\lambda h_v)\leq \frac{\lambda}{\lambda-1/2}\Lip(v)\leq 1+\frac{\delta}{3},\qquad \|\lambda h_v+v\|_{L^\infty} \leq \frac{\delta}{3}.
\]

By Stone-Weierstrass theorem, the smooth function on $\C P^N$ are dense in $C(Y)$. So we can choose such $g_0\in C^\infty(\C P^N)$ such that $\|g_0-v\|_{L^\infty}<\min\{\delta/3,\delta/(3\lambda C_\lambda)\},$ and set $ g=-\lambda g_0.$ Then we have 
\begin{align*}
 \|h-v\|_{L^\infty}
 &\leq\|h+\lambda h_v\|_{L^\infty}
          +\|\lambda h_v+v\|_{L^\infty}\leq\lambda C_\lambda\|g_0-v\|_{L^\infty}+ \frac{\delta}{3}\leq \frac{3\delta}{3},\\
 \Lip(h)
 &\leq\Lip(\lambda h_v)
          +\lambda C_\lambda\|g_0-v\|_\infty<1+2\delta/3.\qedhere
\end{align*}
\end{proof}

Next, we show that the same method can be carried out also on $(X,\omega_\varepsilon)$, thanks to the previous uniform estimates. The next two lemmas use only a Ricci-potential identity, a Sobolev
inequality, and first-derivative bounds. We state these assumptions
explicitly so that the lemmas also apply to the K\"ahler--Ricci flow. And in the next two lemmas, we use $\omega$ to denote any K\"ahler form on $X$ with the assumptions stated in its statement.

\begin{lemma}\label{lem:gradient-estimate-on-total-space}
  Let $\omega$ vary among K\"ahler metrics on $X$ and put
$\mu=\omega^n/\int_X\omega^n$. Suppose, with common constants,
\[
 \begin{gathered}
 \Ric(\omega)=-\chi-dd^cF,\qquad \chi\geq0,\qquad
 \tr_\omega\chi+|\nabla F|_\omega^2\leq C,\\
 \|a\|_{L^{2q}(\mu)}^2
       \leq C_S\int_X(|\nabla a|_\omega^2+a^2)\,d\mu,\qquad q>1.
 \end{gathered}
\]
Fix $\lambda>0$. If
$\Delta_\omega H=\lambda H+G$ and
$\|G\|_{L^\infty}+\|\nabla G\|_{L^\infty}\leq C$, then
\[
 \|H\|_{L^\infty}+\|\nabla H\|_{L^\infty}\leq C'.
\]
The constant $C'$ depends only on $X,C$ and $\lambda$.
\end{lemma}
\begin{proof}
    The maximum principle and the test function $H$ imply
\begin{equation}\label{eq:total-resolvent-energy}
\|H\|_\infty\leq\lambda^{-1}\|G\|_\infty,\qquad \int_X(|\nabla H|^2+\lambda H^2)d\mu=-\int_XGHd\mu\leq\lambda^{-1}\|G\|_\infty^2.
\end{equation}
Write $u=|\nabla H|_\omega^2$. At an $\omega$-unitary frame,
with repeated indices summed, the Bochner identity is
\begin{equation}\label{eq:resolvent-bochner}
 \Delta_\omega u
 =|H_{ij}|^2+|H_{i\bar j}|^2+2\lambda u+2\operatorname{Re}(G_iH_{\bar i})
      -\chi_{i\bar j}H_jH_{\bar i}
      -F_{i\bar j}H_jH_{\bar i}.
\end{equation}
For a nonnegative real test function $\psi$, integration in the
$\bar j$ direction gives the invariant identity
\[
 \begin{split}
 I_\psi&:=\int_X\psi F_{i\bar j}H_jH_{\bar i}\,d\mu\\
 &=-\operatorname{Re}\int_XF_i\bigl(\psi_{\bar j}H_jH_{\bar i}+\psi H_{j\bar j}H_{\bar i} +\psi H_jH_{\bar i\bar j}\bigr)\,d\mu .
 \end{split}
\]
Since $|\nabla F|+|\Delta H|\leq C$,
\begin{equation}\label{eq:resolvent-F-term}
 |I_\psi|
 \leq C\int_X\bigl(u|\nabla\psi|+\psi\sqrt u+\psi\sqrt u\,|H_{ij}|\bigr)\,d\mu.
\end{equation}

Put $v=1+u$ and take $\psi=v^{p-1}$, $p\geq2$.
The two terms requiring absorption satisfy
\[
 \begin{split}
 C(p-1)v^{p-1}|\nabla v|
 &\leq\tfrac{p-1}{4}v^{p-2}|\nabla v|^2+Cp\,v^p,\\
 Cv^{p-1/2}|H_{ij}|
 &\leq\tfrac12v^{p-1}|H_{ij}|^2+Cv^p.
 \end{split}
\]
Also $\chi_{i\bar j}H_jH_{\bar i}\leq Cu$ and
$|G_iH_{\bar i}|\leq C\sqrt u$.
Testing \eqref{eq:resolvent-bochner}, using
\eqref{eq:resolvent-F-term}, and retaining the positive Hessian term
therefore gives
\[
 (p-1)\int v^{p-2}|\nabla v|^2
\leq Cp\int v^p,\qquad \int|\nabla v^{p/2}|^2\leq Cp^2\int v^p.
\]
The common Sobolev inequality yields
$\|v\|_{L^{pq}}\leq(Cp^2)^{1/p}\|v\|_{L^p}$.
Iteration with $p=2q^j$ gives $\|v\|_{L^\infty}\leq C\|v\|_{L^2}$.
Finally
\[\|v\|_{L^2}^2\leq\|v\|_{L^\infty}\|v\|_{L^1},\qquad
 \|v\|_{L^1}=1+\int ud\mu\leq C
\]
by \eqref{eq:total-resolvent-energy}. This bounds $\|v\|_{L^\infty}$.
\end{proof}

We want to show the sharp estimates for the gradient. The idea is similar to the those in section 3, that using Moser iteration to improve weak convergence to strong convergence. The difference is that we need to handle the term involving $\ddc F$ in the differential inequality of the quantity. 
\begin{lemma}\label{lem:gradient-weak-to-sharp}
Under the hypotheses in the preceding lemma, let $c\geq0$ be fixed and set
$z=(|\nabla H|_\omega^2-c)_+$. Then
\[
 \|z\|_{L^\infty}
 \leq C_c\left(\int_Xz\,d\mu\right)^\alpha,\qquad
 \alpha=\frac{q-1}{3q-1}>0.
\]
In particular, $\int_Xz\,d\mu\to0$ implies $\|z\|_{L^\infty}\to0$.
\end{lemma}
\begin{proof}
The preceding lemma gives $0\leq u\leq C$ and $0\leq z\leq C$.
Test \eqref{eq:resolvent-bochner} with $\psi=z^{p-1}$, $p\geq3$.
Use regularized positive parts and pass to the limit if necessary.
On $\{z>0\}$, $\nabla z=\nabla u$.
The first term in \eqref{eq:resolvent-F-term} is bounded by
\[
 C(p-1)z^{p-2}|\nabla z|\leq\tfrac{p-1}{4}z^{p-2}|\nabla z|^2+Cp\,z^{p-2}.
\]
The last term is absorbed by
\[
 Cz^{p-1}|H_{ij}|\leq\tfrac12z^{p-1}|H_{ij}|^2+Cz^{p-1}.
\]
All other terms are bounded by $Cz^{p-1}$.
Since $z^{p-1}+z^p\leq C z^{p-2}$ with a constant independent of
$p$, we obtain
\begin{equation}\label{eq:gradient-excess-energy}
 \int_X\bigl(|\nabla z^{p/2}|^2+z^p\bigr)\,d\mu\leq Cp^2\int_Xz^{p-2}\,d\mu.
\end{equation}
Thus
\[
\|z\|_{L^{pq}}\leq(Cp^2)^{1/p}\|z\|_{L^{p-2}}^{(p-2)/p}.
\]
Let $p_0=1$ and $p_{j+1}=q(p_j+2)$. Applying the last inequality
with $p=p_j+2$ gives
\[
 \|z\|_{L^{p_{j+1}}}\leq[C(p_j+2)^2]^{1/(p_j+2)}\|z\|_{L^{p_j}}^{p_j/(p_j+2)}.
\]
Here
\[
p_j=q^j\left(1+\frac{2q}{q-1}\right)-\frac{2q}{q-1},\qquad
 \prod_{j=0}^{k-1}\frac{p_j}{p_j+2}=\frac{q^k}{p_k}\longrightarrow\alpha .
\]
The product of the constants is bounded because
$\sum_j\log(C(p_j+2)^2)/(p_j+2)<\infty$.
Letting $k\to\infty$ proves the estimate.
\end{proof}

Next, we fix $\lambda>0$, and a smooth function $g\in C^\infty(\C P^N)$. Let $h_\varepsilon$ be the solution to $\Delta_{\oep}h_\varepsilon=\lambda h_\varepsilon+ f^*g$. The trace bound gives
$\|f^*g\|+|\nabla f^*g|_{\oep}\leq C$.
All hypotheses of Lemma~\ref{lem:gradient-estimate-on-total-space}
hold by Sections~2--3; hence
$|h_\varepsilon|+
  |\nabla h_\varepsilon|_{\oep}\leq C$. Let $h\in W^{1,2}$ be the unique solution to $\Delta_{\ocan}h=\lambda h+g$. 

\begin{lemma}
    For any $p\in [1,\infty)$, $\|h_\varepsilon-f^*h\|_{L^p(X,\Omega)}\to 0$
\end{lemma}

\begin{proof}
    Maximum principle and testing with $h_\varepsilon$ give $\|h_\varepsilon\|_{L^\infty}\leq \lambda^{-1}\|g\|_{L^\infty}$ and $\int_X(|\nabla h_\varepsilon|_{\oep}^2+\lambda h_\varepsilon^2)d\mu_\varepsilon\leq \lambda^{-1}\|g\|_{L^\infty}.$
So we can apply Lemma~\ref{lem:compactness} to conclude that passing to a subsequence, $h_\varepsilon\to f^*v$ in every $L^p(X,\Omega)$, and $\|v\|_{L^\infty}<\infty.$ For $\varphi\in C^\infty_c(Y^\circ)$, then we have \[
\int_Xh_\varepsilon(\Delta_{\oep} f^*\varphi-\lambda f^*\varphi)d\mu_\varepsilon=\int_X f^*(g\varphi)d\mu_\varepsilon.
\]
Since $\varphi$ is supported in a compact subset of $Y^\circ$, the metric $\omega_\varepsilon\to f^*\ocan+\varepsilon\osrf$ uniformly there. Therefore, \[
\int_X f^*v(f^*(\tr_{\ocan}\ddc\varphi)-\lambda f^*\varphi)e^{f^*\varphi_{\rm can}}\Omega=\int_Xf^*(g\varphi)e^{f^*\varphi_{\rm can}}\Omega,
\]
which implies  \[
\int_Yv(\Delta_{\ocan}\varphi-\lambda\varphi)\ocan^m=\int_Yg\varphi\ocan^m.
\]
Interior elliptic regularity implies $v\in C^\infty_{loc}(Y^\circ)$ and satisfies $\Delta_{\ocan} v=\lambda v+g$. Since $v\in L^\infty$, by Lemma~\ref{lem:removable}, $v$ can be extended to $Y$ and $v\in W^{1,2}$ and satisfies the global weak equation. From the uniqueness of solution, every subsequence has the same limit, and the lemma follows. \end{proof}

\begin{lemma}\label{lem:convergence-approximation} \leavevmode
\begin{enumerate}
        \item $\|h_\varepsilon-f^*h\|_{L^\infty}\to 0$. 
        \item For any $K\Subset Y^\circ$, \[
        \int_{f^{-1}(K)}|\nabla (h_\varepsilon-f^*h)|_{\oep}^2d\mu_\varepsilon\to 0. 
        \]
\end{enumerate}
\end{lemma}

\begin{proof}
For simplicity, put $e_\varepsilon:=h_\varepsilon-f^*h$.

(1)
Let $d_{FS}$ be the ambient Fubini-Study metric restricted to $Y$. By lemma~\ref{lem:trace} and the normalization $f^*\omega_{FS}=\ell\chi,$ \[
    0\leq f^*\omega_{FS}\leq \ell (\tr_{\oep}\chi)\oep\leq \ell C\oep.
    \]
    Therefore, for any cuver $\gamma$ in $X$, aftering taking infumum, \[
    d_{FS}(f(x),f(x'))\leq \inf_\gamma \operatorname{Length}_{FS}(f\circ\gamma)\leq \inf_{\gamma}\sqrt{\ell C}\operatorname{Length}_{\oep}(f\circ\gamma)=\sqrt{\ell C}d_{\oep}(x,x').
    \]
Apply Lemma~\ref{lem:gradient-estimate-on-total-space}, we get $|\nabla h_\varepsilon|_{\oep}\leq C$ for a uniform constant $C$.

Let $r=d_{\oep}(x,x')$, then \[
|e_\varepsilon(x)-e_\varepsilon(x')|\leq Cr+\sup_{\substack{y,y'\in Y\\d_{FS}(y,y')\leq\sqrt{\ell C}\,r}}|h(y)-h(y')|
\]
Since $(\overline{Y^\circ},\dcan)$ is homeomorphic to $(Y,d_{FS})$, $h$ is continuous on $Y$. Hence, by compactness of $Y$, for each $\delta>0$, there exits $r_\delta\in (0,1)$, which is independent of 
$\varepsilon$, such that \[
|e_\varepsilon(x)|\ge \delta \quad\Longrightarrow\quad
 |e_\varepsilon|\geq\delta/2
 \text{ on }B_{\omega_\varepsilon}(x,r_\delta).
\]
By the volume noncollapsing estimate in Theorem~\ref{thm:bounds}, \[
\|e_\varepsilon\|_{L^2(\mu_\varepsilon)}^2\geq\frac{\delta^2}{4}\mu_\varepsilon(B_{\oep}(x,r_\delta))\geq\frac{\delta^2r_\delta^{a_{\rm vol}}}{4C}>0.
\]
However, Lemma~\ref{lem:convergence-approximation} gives $e_\varepsilon\to 0$ in $L^2(\Omega)$, and also in $L^2(\mu_\varepsilon)$ since $|F_\varepsilon|\leq C$. 

Hence $\|e_\varepsilon\|_{L^\infty}\to 0.$ 

(2) Choose $\xi\in C^\infty_c(Y^\circ)$ with $0\leq\xi\leq1$
and $\xi=1$ near $K$; write $\xi$ also for $f^*\xi$.
On $\operatorname{supp}\xi$, one has \[
\Delta_{\oep}e_\varepsilon=\lambda e_\varepsilon+f^*\Delta_{\ocan}h-\Delta_{\oep}h.
\]
$\tr_{\oep}\chi\leq C$ implies $|\nabla\xi|_{\oep}\leq C, |\Delta_{\oep}f^*h|\leq C$, $|\Delta_{\oep}e_\varepsilon|\leq C$. Therefore \[\begin{aligned}
\int_X\xi^2|\nabla e_\varepsilon|_{\oep}^2d\mu_\varepsilon=&-\int_X\xi^2e_\varepsilon\Delta_{\oep}e_\varepsilon d\mu_\varepsilon-2\int_X\xi e_\varepsilon\langle\nabla\xi,\nabla e_\varepsilon\rangle_{\oep}d\mu_\varepsilon\\
\leq& C\|e_{\varepsilon}\|_{L^\infty}+\frac{1}{2}\int_X\xi^2|\nabla e_\varepsilon|_{\oep}^2d\mu_\varepsilon+2\int_Xe_\varepsilon|\nabla\xi|_{\oep}^2d\mu_\varepsilon.
\end{aligned}
\]
As a reusult, \[
\int_{f^{-1}(K)}|\nabla e_\varepsilon|_{\oep}^2d\mu_\varepsilon\leq C(\|e_\varepsilon\|_{L^\infty}+\|e_\varepsilon\|_{L^\infty}^2)\to 0. 
\]
\end{proof}

Now we can show the sharp gradient estimate. \begin{lemma}\label{lem:sharp-gradient}
$\limsup_{\varepsilon\to0}\Lip_{d_{\oep}}h_\varepsilon\leq\Lip_{d_{\rm can}}(h)$.
\end{lemma}
\begin{proof}
    Let $L_h=\Lip_{\dcan}(h)$ and fix $K\Subset Y^\circ$. Note that our notation is $|\nabla h|_{\omega}^2=\frac{1}{2}|\nabla h|_{g_{\omega}}^2$ where $g_\omega$ is the Riemannian metric of $\omega$. From the sharp Schwarz type estimate, \[
    |\nabla f^*h|_{\omega_\varepsilon}^2\leq (1+o_K(1))f^*|\nabla h|_{\ocan}^2\leq L_h^2/2+o_K(1)
\quad \text{ on } f^{-1}(K).    \]
Note that \[
(|\nabla h_\varepsilon|^2-L_h^2/2)_+\leq |\nabla e_\varepsilon|_{\oep}^2+2|\nabla e_\varepsilon|_{\oep}|\nabla f^* h|_{\oep}+(|\nabla f^*h|_{\oep}^2-L_h^2/2)_+.
\]
By Cauchy-Schwarz inequality and previous lemma, \[
\int_{f^{-1}(K)}(|\nabla h_\varepsilon|^2-L_h^2/2)_+d\mu_\varepsilon\leq \int_{f^{-1}(K)}|\nabla e_\varepsilon|_{\oep}^2d\mu_\varepsilon+C_K\left(\int_{f^{-1}(K)}|\nabla e_\varepsilon|^2d\mu_\varepsilon\right)^{1/2}+o_K(1)\longrightarrow0
\]
Since $\|e_\varepsilon\|_{L^\infty}\to 0$, and $|\nabla h_\varepsilon|_{\oep}$ is uniformly bounded, \[
\limsup_{\varepsilon\to0}\int_X(|\nabla h_\varepsilon|^2-L_h^2/2)_+d\mu_\varepsilon\leq C\Omega\bigl(X\setminus f^{-1}(K)\bigr).
\]
Let $K\uparrow Y^\circ$. Since $\Omega(f^{-1}(D))=0$,
\[\int_X(|\nabla h_\varepsilon|^2-L_h^2/2)_+d\mu_\varepsilon\longrightarrow0.
\]

Thus with Lemma~\ref{lem:gradient-weak-to-sharp}, we get the desired estimate. 
\end{proof}

\begin{proof}[Proof of Lemma~\ref{lem:lip-lift-approximation}]
    Lemma~\ref{lem:approx-lip} provides $g\in C^\infty(\C P^N)$, such that the solution $\Delta_{\ocan}h=\lambda h+g$ satisfies $\|h-v\|_{L^\infty}<\delta$, and $\Lip(h)<1+\delta$. Set $v_\varepsilon=h_\varepsilon$, where $h_\varepsilon$ is the solution to $\Delta_{\oep}h_\varepsilon=\lambda h_\varepsilon+f^*g$. The preceding lemma shows the desired approximation property. 
\end{proof}
\begin{proof}[Completion of the proof of Theorem~\ref{thm:main}]
The case $m=0$ was proved in Section~1. For $m>0$,
Lemma~\ref{lem:metric-convergence} proves the local metric assertion.
Lemma~\ref{lem:lip-lift-approximation} gives
Proposition~\ref{prop:lower-distance}; Proposition~\ref{prop:distortion}
then proves Gromov--Hausdorff convergence.
\end{proof}

\section{A new proof of Gromov-Hausdorff convergence for K\"ahler--Ricci flow}

Based on our previous discussion on the Lipschitz approximation, we can give an independent new proof of the theorem of Lee-Tosatti-Zhang \cite{LTZ26} on the Gromov-Haudorff convergence of the long-time K\"ahler--Ricci flow on compact K\"ahler manifolds with semi-ample canonical line bundle.

Let $\mathchanged{\tilde\omega}(t)$ solve the normalized K\"ahler--Ricci flow
\begin{equation}\label{eq:krf-flow}
 \partial_t\mathchanged{\tilde\omega}(t)=-\Ric(\mathchanged{\tilde\omega}(t))-\mathchanged{\tilde\omega}(t),\qquad
 \mathchanged{\tilde\omega}(0)=\mathchanged{\omega}.
\end{equation}
We retain $f:X\to Y$, $m=\dim_{\C}Y$, $d=n-m$, $\chi$, and $\Omega$
from Section 2. In particular,
$\Ric(\Omega)=-\chi$ and $\int_X\Omega=1$.
The cohomology class is
\[
 [\mathchanged{\tilde\omega}(t)]=(1-e^{-t})2\pi c_1(K_X)+e^{-t}[\mathchanged{\omega}].
\]
Write
\begin{equation}\label{eq:krf-normalization}
 \begin{gathered}
 \chi(t)=(1-e^{-t})\chi+e^{-t}\mathchanged{\omega},\qquad
 \mathchanged{\tilde\omega}(t)=\chi(t)+dd^c\mathchanged{\tilde\varphi}(t),\\
 \mathchanged{\dot{\tilde\varphi}}=\log\frac{e^{dt}\mathchanged{\tilde\omega}(t)^n}{\binom{n}{m}V_FV_Y\Omega}-\mathchanged{\tilde\varphi},
 \qquad \mathchanged{\tilde\varphi}(0)=0,\\
 V(t)=\int_X\mathchanged{\tilde\omega}(t)^n=\int_X\chi(t)^n,\qquad
 d\mu_t=\frac{\mathchanged{\tilde\omega}(t)^n}{V(t)}=e^{F(t)}\Omega.
 \end{gathered}
\end{equation}
Thus
\begin{equation}\label{eq:krf-ricci-potential}
 F(t)=\mathchanged{\dot{\tilde\varphi}}(t)+\mathchanged{\tilde\varphi}(t)-\log(e^{dt}V(t)),\qquad
 \Ric(\mathchanged{\tilde\omega}(t))=-\chi-dd^cF(t).
\end{equation}
We summarize some facts along the NKRF in the next lemma.
\begin{lemma}\label{lem:krf-uniform}
There exist $q>1$ and constants independent of $t\geq0$ such that:
\begin{enumerate}
\item\label{item:krf-density-gradient}
\[
 C^{-1}\leq e^{F(t)}\leq C,\qquad\tr_{\mathchanged{\tilde\omega}(t)}\chi+|\nabla F(t)|_{\mathchanged{\tilde\omega}(t)}^2\leq C.
\]
\item\label{item:krf-sobolev}
For every $v\in W^{1,2}(X,\mathchanged{\tilde\omega}(t))$,
\begin{equation}\label{eq:krf-sobolev}
 \left(\int_X|v|^{2q}\,d\mu_t\right)^{1/q}
 \leq C\int_X\bigl(|\nabla v|_{\mathchanged{\tilde\omega}(t)}^2+v^2\bigr)\,d\mu_t.
\end{equation}
\item\label{item:krf-volume-diameter}
For $x\in X$ and $0<r\leq1$,
\begin{equation}\label{eq:krf-volume-diameter}
 \mu_t(B_{\mathchanged{\tilde\omega}(t)}(x,r))\geq C^{-1}r^{2q/(q-1)},\qquad
 \diam(X,\mathchanged{\tilde\omega}(t))\leq C.
\end{equation}
\end{enumerate}
\end{lemma}

\begin{proof}
(1) The estimates of Song--Tian
\cite[Lemma~2.1 and Propositions~2.1, 2.2, 3.1]{ST16} give
\[
 |\mathchanged{\tilde\varphi}|+|\mathchanged{\dot{\tilde\varphi}}|\leq C,\qquad
 \tr_{\mathchanged{\tilde\omega}(t)}\chi\leq C,\qquad
 |\nabla(\mathchanged{\dot{\tilde\varphi}}+\mathchanged{\tilde\varphi})|_{\mathchanged{\tilde\omega}(t)}^2\leq C.
\]
The last inequality is \cite[equation~(3.19)]{ST16}, obtained by
the maximum principle.
The cohomological volume satisfies
\[
 V(t)=\sum_{k=0}^m\binom nk(1-e^{-t})^k e^{-(n-k)t}
                          \int_X\chi^k\wedge\mathchanged{\omega}^{n-k},
 \qquad C^{-1}\leq e^{dt}V(t)\leq C.
\]

(2) Fix a smooth reference K\"ahler form with volume one.
Part (1) implies that the density of $\mu_t$ relative to its volume
form is bounded above and below by fixed positive constants.
Consequently all its Nash--Yau entropies are bounded, and the lower
density function in \cite[Theorem~2.1]{GPSS23} can be chosen to be
a positive constant. Moreover,
\[
 [\mathchanged{\tilde\omega}(t)]\cdot[\mathchanged{\omega}]^{n-1}
 =(1-e^{-t})A\cdot B^{n-1}+e^{-t}B^n\leq C.
\]
The cited theorem gives a uniform mean-zero Sobolev inequality for
some $q>1$. Since $\mu_t(X)=1$ and
$|\int_Xv\,d\mu_t|^2\leq\int_Xv^2\,d\mu_t$, it implies
\eqref{eq:krf-sobolev}.

(3) It is a standard consequence of the uniform Sobolev inequality from part (2). 
\end{proof}

Fix $\lambda>0$ and $g\in C^\infty(\C P^N,\mathbb R)$.
Define $h_t, h$ to be the unique solution to 
\begin{equation}\label{eq:krf-resolvents}
\Delta_{\mathchanged{\tilde\omega}(t)}h_t=\lambda h_t+f^*g,
 \qquad \Delta_{\ocan}h=\lambda h+g.
\end{equation}
The trace bound $\tr_{\mathchanged{\tilde\omega}(t)}\chi\leq C$ gives
\[
\|f^*g\|_{L^\infty}+\|\nabla^{\mathchanged{\tilde\omega}(t)}(f^*g)\|_{L^\infty}\leq C.
\]
By Lemmas~\ref{lem:krf-uniform} and~\ref{lem:gradient-estimate-on-total-space}(note that all properties we used for $\oep$ in Lemma~\ref{lem:gradient-estimate-on-total-space} are $\tr_{\oep}\chi\le C$ and $|\nabla F_\varepsilon|_{\oep}\leq C$),
we have
\begin{equation}\label{eq:krf-resolvent-bound}
 \|h_t\|_{L^\infty}+\|\nabla^{\mathchanged{\tilde\omega}(t)} h_t\|_{L^\infty}\leq C.
\end{equation}
The constants may depend on the fixed $\lambda$ and $g$.

\begin{lemma}\label{lem:krf-resolvent-limit}
As $t\to\infty$,
\[
 h_t\longrightarrow f^*h
 \quad\text{locally uniformly on }X^\circ
 \quad\text{and in }L^p(X,\Omega),\qquad 1\leq p<\infty.
\]
\end{lemma}

\begin{proof}
Take any sequence $t_j\to\infty$.
On a compact subset of $X^\circ$, the bound
$\mathchanged{\tilde\omega}(t_j)\leq C_K\mathchanged{\omega}$ and \eqref{eq:krf-resolvent-bound}
give equicontinuity in the fixed smooth metric.
Arzel\`a--Ascoli and a compact exhaustion give a subsequence
converging locally uniformly to a bounded continuous function.
On each fiber over $K\Subset Y^\circ$,
\[
 \operatorname{osc}_{X_y}h_{t_j}
 \leq C\diam(X_y,\mathchanged{\tilde\omega}(t_j)|_{X_y})\to 0.
\]
Thus the limit is $f^*h_\infty$, for a bounded continuous function
$h_\infty$ on $Y^\circ$. 

For $\psi\in C^\infty_c(Y^\circ)$, integration by parts gives
\[
 \int_X h_{t_j}
   \bigl(\lambda f^*\psi-\Delta_{\mathchanged{\tilde\omega}(t_j)}f^*\psi\bigr)d\mu_{t_j}
 =\int_X f^*(g\psi)d\mu_{t_j}.
\]
All integrands are supported in a fixed compact subset of $X^\circ$.
The local uniform convergence implyies
\[
 \int_Y h_\infty
     (\lambda\psi-\Delta^{\mathbb R}_{\omega_{\rm can}}\psi)d\nu_{\rm can}
 =\int_Y g\psi\,d\nu_{\rm can}.
\]
Interior elliptic regularity gives $h_\infty\in C^\infty(Y^\circ)$.
By Lemma~\ref{lem:removable}, $h_\infty=h$.
Every subsequential limit is therefore $f^*h$, which proves local
uniform convergence for the whole family. Since $h_t,h$ are bounded
and $\Omega(X\setminus X^\circ)=0$, dominated convergence gives the
finite-$L^p$ limits.
\end{proof}

\begin{lemma}\label{lem:krf-resolvent-energy}
Put $e_t=h_t-f^*h$. Then:
\begin{enumerate}
\item\label{item:krf-resolvent-uniform}
$\|e_t\|_\infty\to0$.
\item\label{item:krf-resolvent-energy}
For every $K\Subset Y^\circ$,
\[
 \int_{f^{-1}(K)}|\nabla e_t|_{\mathchanged{\tilde\omega}(t)}^2\,d\mu_t\longrightarrow0.
\]
\end{enumerate}
\end{lemma}

\begin{proof}
The proof is identical to the proof of \ref{lem:convergence-approximation}.
\end{proof}

\begin{lemma}
\[
\limsup_{t\to\infty}\Lip_{\mathchanged{\tilde\omega}(t)}(h_t)
 \leq\Lip_{d_{\rm can}}(h).
\]
\end{lemma}

\begin{proof}
Just identical to the proof of Lemma~\ref{lem:approx-lip}.
\end{proof}

The lower distance estimate replaces
the reduced-length argument of \cite[Sections~4.2--5]{LTZ26}.
The inputs used here are the Song--Tian flow estimates,
the static Sobolev inequality, regular-locus convergence, and the analysis on the 
canonical model; neither
reduced length nor reduced-volume monotonicity enters this proof.

\bibliographystyle{amsplain}
\bibliography{ref}
\end{document}